\documentclass{amsart}

\usepackage{amsmath,latexsym,amssymb,amsthm,graphicx}
\usepackage{hyperref,doi}
\usepackage{stmaryrd}
\usepackage{kbordermatrix}
\usepackage[all]{xy}
\usepackage{enumitem}
\setenumerate[1]{label=\tn{(\thesection.\arabic*)}}

\numberwithin{equation}{section}

\newtheorem{theorem}{Theorem}[section]
\newtheorem{lemma}[theorem]{Lemma}

\newtheorem{corollary}[theorem]{Corollary}
\newtheorem{proposition}[theorem]{Proposition}

\theoremstyle{definition}
\newtheorem{remark}[theorem]{Remark}

\newtheorem{example}[theorem]{Example}

\newtheorem{question}[theorem]{Question}

\newtheorem{conjectures}[theorem]{Conjectures}

\def\N{\ensuremath{\mathbb{N}}}
\def\Z{\ensuremath{\mathbb{Z}}}
\def\Q{\ensuremath{\mathbb{Q}}}

\newcommand{\pa}[1]{\left(#1\right)}
\newcommand{\cpa}[1]{\left\{#1\right\}}
\newcommand{\wt}[1]{\widetilde{#1}}
\newcommand{\tn}[1]{\textnormal{#1}}
\newcommand{\br}[1]{\left[#1\right]}
\newcommand{\fg}[1]{\left\langle #1\right\rangle}

\newcommand{\cov}[1]{\tn{Cov}\pa{#1}}

\def\fa{f^{\ast}}
\def\fs{f_{\sharp}}
\newcommand{\bs}[1]{\pa{\uppercase{#1},{\lowercase{#1}}_{0}}}
\newcommand{\p}[1]{\pi_1{\bs{#1}}}
\newcommand{\im}[1]{\tn{Im} \, #1}

\newcommand{\rest}[1]{\left.#1\right|}
\newcommand{\card}[1]{\left|#1\right|}

\newcommand{\burn}[1]{\mathcal{B}\hspace{-.2em}\pa{#1}}
\newcommand{\gset}[1]{{#1}\tn{-\textbf{set}}}

\def\ind{\tn{Ind}}
\def\res{\tn{Res}}
\def\cov{\tn{Cov}}
\def\b{\backslash}
\def\s{\slash}

\font\cuf=cmtt8
\newcommand{\curl}[1]{{\cuf #1}}

\begin{document}
\title{Topological and algebraic pullback functors}

\author[J.S.~Calcut]{Jack S. Calcut}
\address{Department of Mathematics\\
         Oberlin College\\
         Oberlin, OH 44074}
\email{jcalcut@oberlin.edu}
\urladdr{\href{http://www.oberlin.edu/faculty/jcalcut/}{\curl{http://www.oberlin.edu/faculty/jcalcut/}}}

\author[J.D.~McCarthy]{John D. McCarthy}
\address{Department of Mathematics\\
                    Michigan State University\\
                    East Lansing, MI 48824-1027}
\email{mccarthy@math.msu.edu}
\urladdr{\href{http://www.math.msu.edu/~mccarthy/}{\curl{http://www.math.msu.edu/\textasciitilde mccarthy/}}}

\author[J.J.~Walthers]{Jeremy J. Walthers}

\date{May 14, 2012}

\begin{abstract}
We give algebraic equivalents for certain desirable properties of pullback functors on categories of coverings and group sets,
namely nullity zero, essential injectivity, and essential surjectivity.
Nullity zero turns out to be equivalent to the notion of a contranormal subgroup.
We observe a Tannakian-like phenomenon with essential injectivity.
Essential surjectivity is intimately related to Zappa-Sz{\'e}p products.
We include several examples, and some open questions.
\end{abstract}

\maketitle

\section{Introduction}\label{s:intro}

Given a continuous function of spaces $f:X\to Y$, the \emph{topological pullback functor} \hbox{$\fa:\cov\pa{Y}\to\cov\pa{X}$} sends coverings of $Y$ to coverings of $X$.
Given an arbitrary group homomorphism $h:H\to G$, we define the \emph{algebraic pullback functor} $h^{\ast}:\gset{G}\to\gset{H}$.
In case $h$ is inclusion of finite groups, $h^{\ast}$ is the \emph{restriction functor} of Burnside ring theory.
Let $\fs$ be the induced homomorphism of fundamental groups.
Topological and algebraic pullback are intimately related by considering $\fs^{\ast}$, as displayed in diagram~\eqref{keydiagram} below.\\

In~\cite{quillen}, Quillen gave algebraic equivalents for $\fa$ to be faithful, full and faithful, and an equivalence of categories.
We call this \emph{Quillen's triad}.
Quillen took $f$ to be a map of posets; for his first two equivalences, the target was a point.
In~\cite{calcutmccarthy}, the first two authors generalize Quillen's triad to much more general topological spaces, without restricting the target to be a point, and to other categories of coverings.
In this paper, we extend Quillen's triad in other, more algebraic directions.
We work with reasonably nice spaces (see Section~\ref{s:cov} for our hypotheses), so classical covering space theory works, and give purely algebraic equivalents for $\fa$ and $h^{\ast}$ to have nullity zero, to be essentially injective, and to be essentially surjective.\\

We say a pullback functor has \emph{nullity zero} provided only the trivial objects pullback to trivial objects (see~Section~\ref{s:nullityzero}).
Following Rose~\cite{rose}, a subgroup $L$ of $G$ is \emph{contranormal} provided the normal closure of $L$ in $G$ equals $G$.
We prove that a pullback functor has nullity zero if and only if the image subgroup is contranormal.
Thus, the algebraic notion of a contranormal subgroup has a topological equivalent.
Namely, let $H$ be a subgroup of a group $G$.
Realize inclusion $H\hookrightarrow G$ as $\fs$ for nice spaces; this is always possible by Lemma~\ref{alg_modelled}.
Then, $H$ is contranormal in $G$ if and only if the topological pullback functor $\fa$ has nullity zero.
We give several examples of contranormal subgroups arising naturally both algebraically and topologically.
In the free group of rank two, we show that there appear to be vastly more contranormal subgroups than normal subgroups for each finite index $n>2$.
This raises the question: \emph{are contranormal subgroups more prevalent than normal subgroups for finite index $n>2$ and most groups in the sense of Gromov?} We hope to explore this question in future work.\\

We prove that a pullback functor is essentially injective if and only if the associated group homomorphism is surjective.
Our proof of the reverse implication is direct, whereas our proof of the forward implication utilizes infinite component covers and an infinite swindle.
Recall the \emph{Tannakian philosophy} from representation theory~\cite{joyal_street}: a group is determined by the category of its finite dimensional representations.
Thus, the Tannakian-like question arises: may failure of essential injectivity be detected using only \emph{finite} component covers of $Y$ (equivalently, $G$-sets having finitely many orbits)?
We answer this question in the affirmative for \emph{arbitrary} groups.
Our proof, in the finite index case, first reduces to the finite group case, which we then solve using Burnside rings.
We give two proofs for the finite case, the first using a lemma of Bouc (Lemma~\ref{bouc_lemma} below).
Our second proof (chronologically our first) identifies a distinguished, $1$-dimensional subspace of the kernel of the restriction functor when $H$ is a proper subgroup of $G$.
The existence of this distinguished subspace permits us to assign a natural number $\Delta(G,H)$ to each finite group and subgroup pair, which we call the \emph{deviation} of $H$ in $G$.
The deviation is an isomorphism invariant of the pair $(G,H)$ and, in fact, depends only on the $G$-conjugacy class of $H$ in $G$.
In case $H$ is normal in $G$, the deviation equals the index $\br{G:H}$.
In general, $\Delta(G,H)$ need not equal $\br{G:H}$, and the two may coincide even when $H$ is not normal in $G$.
We conjecture that $\Delta(G,H)=1$ if and only if $H=G$, that $\br{G:H}$ divides $\Delta(G,H)$, and that $\Delta(G,H)$ divides the order of $G$.
We present some evidence for these conjectures.\\

Understanding the kernel of a general morphism is sometimes equivalent to understanding injectivity.
Our results on nullity zero and essential injectivity show that this is decidedly not the case with pullback functors.
Namely, a pullback functor may have nullity zero while failing to be essentially injective (see examples in Section~\ref{s:nullityzero}).\\

We give an algebraic equivalent for a pullback functor to be essentially surjective.
Our equivalence imposes, for each subgroup $K$ of $H$, a constraint on the pair $(G,H)$ being essentially surjective.
Taking $K$ to be trivial yields the necessary, but generally not sufficient, condition: $G$ must split as a Zappa-Sz{\'e}p product of $H$ and a subgroup $L$ of $G$ (called a \emph{complement} of $H$ in $G$).
Zappa-Sz{\'e}p products generalize semidirect products.
We present a positive class of examples that are essentially surjective (they are special semidirect products).
We further show, by explicit example, that this class does not encompass all essentially surjective pairs.
We leave open the question of which subgroups of $H$ yield interesting constraints on essential surjectivity.\\

This paper is organized as follows.
Section~\ref{s:cov} recalls topological pullback and fixes some notation.
Section~\ref{s:Gset} defines algebraic pullback and proves some properties of algebraic and topological pullback.
Section~\ref{s:nullityzero} studies nullity zero,
Section~\ref{s:essinj} studies essential injectivity, and
Section~\ref{s:esssurj} studies essential surjectivity.\\

Throughout, $\N:=\cpa{1,2,3,\ldots}$ denotes the natural numbers.
$\card{S}$ denotes the cardinality of $S$.
Define $\omega:=\card{\N}$.
$K<L$ means that $K$ is a (not necessarily proper) subgroup of $L$.
A proper subgroup of $G$ is any subgroup $L\lneqq G$ ($L=\cpa{e}$ permitted in case $G\neq \cpa{e}$).
A functor $F:C\to D$ is \textbf{essentially injective} provided: if $F(x)\cong F(y)$, then $x\cong y$.
$F$ is \textbf{essentially surjective} provided: if $d$ is an object in $D$, then there exists an object $c$ in $C$ such that $F(c)\cong d$.

\section{Coverings and Pullback}\label{s:cov}

Fix a \textbf{map} (= continuous function) $f:X\to Y$ of topological spaces.
We assume $X$ and $Y$ are connected, locally path-connected, and semilocally simply-connected.
Spaces are not required to be Hausdorff.
Indeed, classical covering space theory `works' without any Hausdorff hypothesis~\cite[Ch.~1]{hatcher}.
Despite the fact that our main interest lies in the unbased category, it will be useful to base spaces.
So, fix some $x_0\in X$ and define $y_0:=f\pa{x_0}$.
Thus, we have the based map:
\[
\xymatrix@R=0pt{
	\bs{X}	\ar[r]^-{f}		&		\bs{Y}}
\]

Recall the category $\cov\pa{Y}$ of unbased coverings of $Y$.
An \textbf{object} of $\cov\pa{Y}$ is an unbased covering $p:E\to Y$ ($E$ may be disconnected or empty).
A \textbf{morphism} from $p_1:E_1\to Y$ to $p_2:E_2\to Y$ is a map $t:E_1\to E_2$ such that $p_1=p_2\circ t$.
Write $E_1\cong E_2$ to mean unbased isomorphism of coverings.\\

As $Y$ is locally path-connected, the restriction of any object $p:E\to Y$ to any union of components of $E$ is also an object of $\cov\pa{Y}$.
As $Y$ is locally path-connected and semilocally simply-connected, the disjoint union of any collection of objects of $\cov\pa{Y}$ is itself an object of $\cov\pa{Y}$. We refer the reader to~\cite{calcutmccarthy} for detailed proofs of basic properties of $\cov\pa{Y}$ and topological pullback.\\

We recall the \emph{topological pullback functor} on coverings:
\[
	\fa:\cov\pa{Y}\to\cov\pa{X}
\]
Let $p:E\to Y$ be an object of $\cov\pa{Y}$.
The \textbf{topological pullback} of $p$ along $f$ consists of the subspace:
\begin{equation}\label{pullback_set}
	\fa\pa{E} := \cpa{ \pa{x,e}\in X\times E \mid f(x)=p(e) }\subset X\times E
\end{equation}
and the commutative diagram:
\begin{equation}\label{pullback}\begin{split}
\xymatrix{
    \fa\pa{E}	\ar[r]^-{\wt{f}}	\ar[d]_{\fa(p)}	&	E	\ar[d]^{p}\\
    X  							\ar[r]^-{f}     						& Y }
\end{split}\end{equation}
Here, $\fa\pa{p}$ and $\wt{f}$ are restrictions of the coordinate projections, and $\fa\pa{p}$ is a covering map.
Note that:
\begin{equation}\label{pbfiber}
	\fa\pa{p}^{-1}\pa{x_0}=\cpa{x_0}\times p^{-1}\pa{y_0}
\end{equation}
and:
\[
\xymatrix@R=0pt{
	\fa\pa{p}^{-1}\pa{x_0}	\ar[r]^-{\wt{f}\,|}				&		p^{-1}\pa{y_0}	\\
	\pa{x_0,z}   		\ar@{|-{>}}[r]    &		z}
\]
is the canonical homeomorphism of fibers.
If $t$ is a morphism from $p_1:E_1\to Y$ to $p_2:E_2\to Y$, then:
\begin{equation}\label{pbmorphism}
	\fa\pa{t}:=\rest{\pa{\tn{id}_{X}\times t}}{\fa\pa{E_1}}
\end{equation}
is a morphism from $\fa\pa{p_1}$ to $\fa\pa{p_2}$.
Thus, \hbox{$\fa:\cov\pa{Y}\to\cov\pa{X}$} is a covariant functor.\\

Disjoint union is denoted $+$ or $\Sigma$.
Pullback respects disjoint union.
Namely, for each index set $S$ and objects $p_i:E_i\to Y$, $i\in S$, of $\cov\pa{Y}$:
\[
	\fa\pa{\sum_{i\in S} E_i}\cong \sum_{i\in S}\fa\pa{E_i}
\]
If $c=\card{S}$ and $p:E\to Y$ is an object of $\cov\pa{Y}$, then define $c\cdot E := \Sigma_{i\in S} E$.
Hence:
\[
\fa\pa{c\cdot E}	\cong	c\cdot\fa\pa{E}
\]

The based map $f$ induces the homomorphism of fundamental groups:
\[
\xymatrix@R=0pt{
	\p{X}	\ar[r]^-{\fs}		&		\p{Y}}
\]
Define:
\begin{align*}
	J&:=\p{X}\\
	G&:=\p{Y}\\
	H&:=\im{\fs} < G\\
	N&:=\ker{\fs} \triangleleft J
\end{align*}
If $L<G$, then we define:
\begin{equation}\label{prime_notation}
	L':=\fs^{-1}(L)<J
\end{equation}
Write $K\equiv_G L$ to mean that $K$ and $L$ are $G$-conjugate subgroups of $G$.
If $L<G$, then $\br{L}:=\cpa{K<G\mid K\equiv_G L}$ is the $G$-conjugacy class of $L$ in $G$.
Define:
\[
	SG:=\cpa{\br{L}\mid L<G}
\]
the set of $G$-conjugacy classes of subgroups of $G$.
Similarly, define $\equiv_J$, $\equiv_H$, and $SH$.
In cases where confusion may arise, we will write $\br{L}_H$ or $\br{L}_G$.\\

Let $\br{L}\in SG$. By the classification of covering spaces, there exists:
\begin{align*}
	Y_{\br{L}} 	&= \tn{the unbased, connected cover of $Y$ corresponding}\\
								&\hspace{14pt} \tn{to $\br{L}$ (unique up to unbased isomorphism)} 
\end{align*}
A \textbf{trivial cover} of $Y$ is any cover isomorphic to $c\cdot Y$ for some cardinal number $c$.
The following abbreviations will be used:
\begin{equation}\label{abbrev}
\begin{alignedat}{2}
	\wt{Y} 	& \ \  \tn{denotes} \ \ 	Y_{\br{\cpa{e}}} 		& \quad & \tn{(the connected, simply-connected cover of $Y$)}\\
	Y				& \ \ \tn{denotes} \ \	Y_{\br{G}} 								& \quad &\tn{(the $1$-sheeted, trivial cover of $Y$)}\\
	\wt{X} 	& \ \ \tn{denotes} \ \ 	X_{\br{\cpa{e}}} 		& \quad & \tn{(the connected, simply-connected cover of $X$)}\\
	X				& \ \ \tn{denotes} \ \	X_{\br{J}} 								& \quad &\tn{(the $1$-sheeted, trivial cover of $X$)}
\end{alignedat}
\end{equation}

Let $p:E\to Y$ be an object of $\cov\pa{Y}$. Then, $E$ is the disjoint union of its components, each of which is isomorphic to $Y_{\br{L}}$ for some $\br{L}\in SG$.
It follows that:
\[
	E\cong \sum_{\br{L}\in SG} c_{\br{L}}\cdot Y_{\br{L}}
\]
for some cardinal numbers $c_{\br{L}}$.
Observe that:
\[
	\sum_{\br{L}\in SG} c_{\br{L}}\cdot Y_{\br{L}} \cong \sum_{\br{L}\in SG} d_{\br{L}}\cdot Y_{\br{L}}
\]
if and only if $c_{\br{L}}=d_{\br{L}}$ for each $\br{L}\in SG$.\\

In the coming sections, we study the topological pullback functor $\fa$ via the intimately related \emph{algebraic pullback functor} associated to $\fs$.
While the discusion turns algebraic, it is helpful to recall that every homomorphism of groups arises as the induced homomorphism on fundamental groups for some decent spaces.

\begin{lemma}\label{alg_modelled}
Let $h:J_0\to G_0$ be an arbitrary homomorphism of groups (no restriction on $\card{J_0}$ or $\card{G_0}$).
Then, there exist connected $2$-dimensional CW-complexes $\bs{X}$ and $\bs{Y}$ and isomorphisms:
\begin{align*}
	\varphi & :J_0\to\p{X}=:J\\
	\psi & :G_0\to\p{Y}=:G
\end{align*}
Further, there exists a map $f:\bs{X}\to\bs{Y}$ such that the following diagram commutes:
\begin{equation}\label{alg_to_topology}\begin{split}
\xymatrix{
    J_0		\ar[r]^{h}	\ar[d]_-{\varphi}^-{\cong}	&	G_0	\ar[d]_-{\psi}^-{\cong}	\\
    J  	\ar[r]^-{\fs} 	& G}
\end{split}\end{equation}
\end{lemma}

\begin{proof}
Consider the multiplication table presentations $\fg{J_0\mid R}$ and $\fg{G_0\mid S}$ of $J_0$ and $G_0$ (see~\hbox{\cite[pp.~7--8]{mks}}).
Let $\bs{X}$ and $\bs{Y}$ be the standard CW-complexes of dimension $2$ associated to $\fg{J_0\mid R}$ and $\fg{G_0\mid S}$ respectively (see \cite[p.~52]{hatcher}).
The construction of these complexes yields the isomorphisms $\varphi$ and $\psi$ (coherent orientation of loops is required).
The obvious function $f:\bs{X}\to\bs{Y}$ is easily seen to be a map as desired.
\end{proof}

\begin{remark}\label{injective_cw_realization}
If $h$ is injective (which will turn out to be the most important case), then an alternative approach to Lemma~\ref{alg_modelled} is as follows.
Begin with any presentation $P$ of $G_0$,
construct the standard $2$-dimensional CW-complex $\bs{Y}$ associated to $P$,
then use covering space theory to get $\bs{X}$ as an appropriate, connected cover of $\bs{Y}$.
The resulting map $f:\bs{X}\to\bs{Y}$ is itself a covering map, so $\bs{X}$ is a $2$-dimensional CW-complex.
\end{remark}

\begin{remark}\label{injective_fp_4mfld_realization}
If $h$ is injective and $G_0$ is finitely presented, then one may use \hbox{$4$-manifolds} in place of CW-complexes.
As is well known~\cite[pp.~131,155]{gompf_stipsicz}, surgery yields a connected, smooth, closed (= compact, no boundary) $4$-manifold
$\bs{Y}$ with $\p{Y}\cong G_0$.
Use covering space theory to get $\bs{X}$ as an appropriate, connected (possibly noncompact) cover of $\bs{Y}$.
The resulting map \hbox{$f:\bs{X}\to\bs{Y}$} is itself a covering map, so $\bs{X}$ is a smooth $4$-manifold.
\end{remark}

\section{G-sets and Pullback}\label{s:Gset}

Let $\gset{G}$ denote the category of (not necessarily finite or nonempty) right \hbox{$G$-sets}.
A \textbf{morphism} of $G$-sets $S_1$ and $S_2$ is a $G$-equivariant function $t:S_1\to S_2$ (i.e., $t\pa{s\cdot g}=t\pa{s}\cdot g$).
The categories $\cov\pa{Y}$ and $\gset{G}$ are equivalent by the functor (``$F$'' for \emph{fiber}):
\begin{equation}\label{Fequiv}\begin{split}
\xymatrix@R=0pt{
      \cov\pa{Y}     \ar[r]^-{F} 						&  \gset{G}  \\
    		p  									\ar@{|-{>}}[r]    &  p^{-1}\pa{y_0}}
\end{split}\end{equation}
Here, $p:E\to Y$, and $G$ acts on the fiber by the \textbf{monodromy action}:
\[
	z\cdot g:=\wt{\gamma}\pa{1}
\]
where $\gamma:([0,1],\cpa{0,1})\to \bs{Y}$ is such that $g=\br{\gamma}$, and $\wt{\gamma}$ is the lift of $\gamma$ to $E$ such that $\wt{\gamma}\pa{0}=z$.
If $t:p_1\to p_2$ is a morphism, then $F\pa{t}$ is, by definition, the restriction $\rest{t}:p_1^{-1}\pa{y_0}\to p_2^{-1}\pa{y_0}$.

\begin{remark}\label{remequiv}
One may construct a weak inverse for $F$ by sending a (discrete) $G$-set $S$ to $(\wt{Y}\times S)\s G$ for a suitable action of $G$ on $\wt{Y}\times S$.
This construction involves choices, and there is no canonical weak inverse for $F$ without additional data.
For our purposes, it is more useful to recall the theorem that a functor $F$ is an equivalence if and only if $F$ is full, faithful, and essentially surjective~\hbox{\cite[p.~93]{maclane}}.
\end{remark}

Let $S$ be a right $G$-set.
If $s\in S$, then $_{s}G<G$ denotes the \textbf{stabilizer} of $s$ and $sG\subset S$ denotes the \textbf{orbit} of $s$.
If $a\cdot g =b$, then $_{b}G = g^{-1} \pa{_{a}G}g$.
The \textbf{orbit space} is $S\s G:=\cpa{sG \mid s\in S}$.
A \textbf{transversal} $T$ for $S\s G$ is a set containing exactly one element from each orbit.\\

If $L<G$, then the set of right cosets $L\b G$ is a transitive $G$-set where $G$ acts by right translation.
If $Lg\in L\b G$, then $_{Lg}G=g^{-1}Lg$.
Given subgroups $L$ and $K$ of $G$, $L\b G\cong K\b G$ (as $G$-sets) if and only if $L\equiv_{G} K$.
Each transitive right $G$-set $S$ is (noncanonically) isomorphic to $L\b G$ for some $L<G$.
Namely, if $s\in S$, then an isomorphism is:
\begin{equation}\label{stabcosettrans}\begin{split}
\xymatrix@R=0pt{
	_{s}G\b G	\ar[r]			&		S	\\
	_{s}G g 		\ar@{|-{>}}[r]    &		s\cdot g}
\end{split}\end{equation}

As in~\eqref{Fequiv}, the functor:
\[
\xymatrix@R=0pt{
      \cov\pa{X}     \ar[r]^-{F} 						&  \gset{J}  \\
    		q  									\ar@{|-{>}}[r]    &  q^{-1}\pa{x_0}}
\]
is an equivalence, where $J$ acts on $q^{-1}\pa{x_0}$ by the $J$-monodromy action.
Consider the diagram of functors:
\begin{equation}\label{keysquare}\begin{split}
\xymatrix{
    \cov\pa{Y}	\ar[r]^-{F}	\ar[d]_-{\fa}	&	\gset{G}	\ar@{--{>}}[d]^{\varepsilon}	\\
    \cov\pa{X} 	\ar[r]^-{F} 								&	\gset{J} }
\end{split}\end{equation}
Commutativity of the pullback square~\eqref{pullback} implies that the $J$-monodromy action on $\fa\pa{p}^{-1}\pa{x_0}$ and the $G$-monodromy action on $p^{-1}\pa{y_0}$ satisfy:
\begin{equation}\label{JGactions}
	\pa{x_0,z}\cdot j	=	\pa{x_0,z\cdot\fs\pa{j}}
\end{equation}
Thus, there is a canonical functor $\varepsilon$ that makes~\eqref{keysquare} commute.
Namely, define $\varepsilon$ on objects by $\varepsilon\pa{S}:=\cpa{x_0}\times S$ where $\pa{x_0,s}\cdot j := \pa{x_0,s\cdot \fs(j)}$, and on morphisms by $\varepsilon(t):=\tn{id}\times t$.
Recalling~\eqref{pbfiber}, \eqref{pbmorphism}, and~\eqref{JGactions}, it is straightforward to verify that~\eqref{keysquare}, with $\varepsilon$ included, is a commutative diagram of functors.\\

A second functor $\gset{G}\to\gset{J}$, closely related to $\varepsilon$ but even more canonical, is what we call the \emph{algebraic pullback functor} associated to $\fs:J\to G$.
We define it now for a general homomorphism.\\

Let $h:G_1\to G_2$ be a homomorphism of groups.
The \textbf{algebraic pullback functor} is:
\[
	h^{\ast}:\gset{G_2}\to\gset{G_1}
\]
defined on objects by $h^{\ast}\pa{S}:=S$ where $s\cdot g_1 :=s\cdot h\pa{g_1}$, and defined on morphisms by $h^{\ast}(t):=t$.
Evidently, algebraic pullback respects disjoint union.

\begin{remark}
If $h$ is inclusion, then $h^{\ast}$ is restriction of the $G_2$ action to $G_1$.
Further, if $G_2$ is finite, then $h^{\ast}$ is typically denoted $\res$ or $\res_{G_1}^{G_2}$ in the literature.
We use the $\res$ notation in Sections~\ref{ss:finitegroups} and~\ref{ss:finite2} ahead with finite groups.
\end{remark}

\begin{lemma}[Basic Properties of Algebraic Pullback]\label{basicpropsalgpb}
Let $h:G_1\to G_2$ be a homomorphism of groups.
Let $I:=\im{h}< G_2$.
Let $S$ be a $G_2$-set and let $s\in S$. Then:
\begin{enumerate}\setcounter{enumi}{\value{equation}}
\item\label{stabilizersgrouppb} The stabilizers satisfy $_{s}G_1 = h^{-1}\pa{_{s}G_2}$.
\item\label{hisoequiv} If $h$ is an isomorphism, then $h^{\ast}$ and $\pa{h^{-1}}^{\ast}$ are inverse functors and, hence, are equivalences.
\item\label{surjorbitspacebij} If $h$ is surjective, then the following is a bijection of orbit spaces:
\begin{equation*}\begin{split}
\xymatrix@R=0pt{
	S\s G_2	\ar[r]			&		S\s G_1\\
	s G_2		\ar@{|-{>}}[r]    &		s G_1}
\end{split}\end{equation*}
\item\label{hsurjifftranstotrans} $h$ is surjective if and only if $h^{\ast}$ sends each transitive $G_2$-set to a transitive $G_1$-set.
\item\label{hast} If $h$ is surjective and $L<G_2$, then $h^{\ast}\pa{L\b G_2} \cong h^{-1}\pa{L}\b G_1$.
\item\label{pbtransG2set} If $L<G_2$, then $\pa{L\b G_2}\s G_1 = L\b G_2\s I$.
Furthermore:
\[
	h^{\ast}\pa{L\b G_2} \cong \sum_{\substack{LgI\in\\ L\b G_2\s I}}  h^{-1}\pa{g^{-1}Lg}\b G_1
\]
\setcounter{equation}{\value{enumi}}
\end{enumerate}
\end{lemma}

\begin{proof}
Items~\ref{stabilizersgrouppb}--\ref{surjorbitspacebij} are exercises.
In~\ref{hsurjifftranstotrans}, the forward direction is immediate by~\ref{surjorbitspacebij}.
For the backward direction, let $L:=\im{h}<G_2$. By hypothesis, $h^{\ast}\pa{L\b G_2}$ is a transitive $G_1$-set.
Let $g_2\in G_2$. Then, there exists $g_1\in G_1$ such that $\pa{Le}\cdot g_1 = Lg_2$.
Hence, $L=Lg_2$ and $g_2\in L$ as desired.
Item~\ref{hast} follows from~\ref{surjorbitspacebij}, \eqref{stabcosettrans}, and~\ref{stabilizersgrouppb}.
For~\ref{pbtransG2set}, note that $\pa{L\b G_2}\s G_1$ is the orbit space for the right $G_1$ action on $h^{\ast}\pa{L\b G_2}$,
and $L\b G_2\s I$ is the set of double cosets of $L$ and $I$ in $G_2$.
It is straightforward to see the two are equal.
Finally, $h^{\ast}\pa{L\b G_2}$ is a disjoint union of transitive $G_1$-sets, namely the individual orbits in $\pa{L\b G_2}\s G_1$.
Let $LgI$ be such an orbit.
By~\eqref{stabcosettrans}, $LgI\cong {_{Lg}G_1}\b G_1$.
By~\ref{stabilizersgrouppb}, $_{Lg}G_1=h^{-1}\pa{_{Lg}G_2}=h^{-1}\pa{g^{-1}Lg}$ as desired.
\end{proof}

\begin{remark}
As an application of Lemma~\ref{basicpropsalgpb}, recall diagram~\eqref{alg_to_topology}. Algebraic pullback yields the commutative diagram of functors:
\begin{equation}\label{alg_to_topology_pb}\begin{split}
\xymatrix{
    \gset{J_0}	&	\ar[l]_(.45){h^{\ast}}		\gset{G_0}		\\
    \gset{J}  	\ar[u]^-{\varphi^{\ast}}	& \ar[l]_(.45){\fs^{\ast}}	\ar[u]_-{\psi^{\ast}}	\gset{G}}
\end{split}\end{equation}
where the vertical functors are equivalences by~\ref{hisoequiv}.
Thus, $h^{\ast}$ and $\fs^{\ast}$ behave identically concerning essential injectivity, essential surjectivity, and nullity zero (defined in Section~\ref{s:nullityzero}).
\end{remark}

\begin{corollary}[Further Properties of Algebraic Pullback]\label{corap}
Let $h:G_1\to G_2$ be a homomorphism of groups.
Let $I:=\im{h}< G_2$.
The following are consequences of~\ref{pbtransG2set}:
\begin{enumerate}\setcounter{enumi}{\value{equation}}
\item\label{hstartG2G2} $h^{\ast}\pa{G_2\b G_2}\cong G_1\b G_1$
\item If $L\triangleleft G_2$, then $h^{\ast}\pa{L\b G_2}\cong c\cdot \pa{h^{-1}\pa{L}\b G_1}$ where $c=\card{L\b G_2\s I}\geq 1$.
\item\label{apbeG2} $h^{\ast}\pa{\cpa{e}\b G_2} \cong \br{G_2:I}\cdot \pa{\ker{h}\b G_1}$.
\item If $h$ is injective, then $h^{\ast}\pa{\cpa{e}\b G_2} \cong \br{G_2:I}\cdot \pa{\cpa{e}\b G_1}$.
\item\label{hstarI} If $I\triangleleft G_2$, then $I\b G_2\s I=I\b G_2$ and $h^{\ast}\pa{I\b G_2}\cong \br{G_2:I}\cdot \pa{G_1\b G_1}$.
\item\label{hstarIG2} $h^{\ast}\pa{I\b G_2}\cong c\cdot\pa{G_1\b G_1} +E$ where $c\geq 1$ and $E$ contains no orbit isomorphic to $G_1\b G_1$ ($E$ may be empty).
\setcounter{equation}{\value{enumi}}
\end{enumerate}
\qed
\end{corollary}

So, $\fs:J\to G$ yields the algebraic pullback functor $\fs^{\ast}:\gset{G}\to\gset{J}$.
By diagram~\eqref{keysquare}, we have the diagram of functors:
\begin{equation}\label{keysquare2}\begin{split}
\xymatrix{
    \cov\pa{Y}	\ar[r]^-{F}	\ar[d]_-{\fa}	&	\gset{G}	\ar@{}[d]|{\Leftrightarrow}	\ar@/^/[d]^{\fs^{\ast}}	\ar@/_/[d]_{\varepsilon}	\\
    \cov\pa{X} 	\ar[r]^-{F} 								&	\gset{J} 		}
\end{split}\end{equation}
While the functors $\varepsilon$ and $\fs^{\ast}$ are not equal, they are naturally isomorphic
(as indicated by the $\Leftrightarrow$ in diagram~\eqref{keysquare2}).
Namely, $\rho:\varepsilon \Rightarrow \fs^{\ast}$ defined by $\rho\pa{S}:=\pa{\pa{x_0,s}\mapsto s}$,
and $\nu:\fs^{\ast} \Rightarrow \varepsilon$ defined by $\nu\pa{S}:=\pa{s\mapsto\pa{x_0,s}}$,
are natural isomorphisms.
In particular, $\rho\pa{S}$ and $\nu\pa{S}$ are isomorphisms of $J$-sets for each object $S$ in $\gset{G}$.\\

The homomorphism $\fs:J\to G$ factors uniquely as a surjection followed by an inclusion:
\begin{equation}\label{factorfs}\begin{split}
\xymatrix@C=1.4pc{
    J	\ar[rr]^-{\fs}	\ar@{->>}[dr]_-{\lambda}	&		&	G	\\
    &	H \ar@{^{(}->}[ur]_-{\iota} }
\end{split}\end{equation}
which yields the commutative diagram of algebraic pullback functors:
\begin{equation}\label{keytriangle}\begin{split}
\xymatrix@C=0pc{
    \gset{J}	&		&	\gset{G}	\ar[ll]_-{\fs^{\ast}}	\ar[dl]^-{\iota^{\ast}}\\
    &	\gset{H} \ar[ul]^-{\lambda^{\ast}} }
\end{split}\end{equation}

Diagrams~\eqref{keysquare2} and~\eqref{keytriangle} yield the key diagram of functors:
\begin{equation}\label{keydiagram}\begin{split}
\xymatrix{
    \cov\pa{Y}	\ar[r]^-{F}	\ar[d]_-{\fa}	&	\gset{G}	\ar@{}[d]|{\Leftrightarrow} \ar[dr]^{\iota^{\ast}}	\ar@/^/[d]^{\fs^{\ast}}	\ar@/_/[d]_{\varepsilon}	\\
    \cov\pa{X} 	\ar[r]^-{F} 								&	\gset{J} 		& \gset{H} \ar[l]_{\lambda^{\ast}}}
\end{split}\end{equation}
In~\eqref{keydiagram}, the left square and the right triangle each commute, both functors labelled $F$ are equivalences,
and the functors $\varepsilon$ and $\fs^{\ast}$ are naturally isomorphic (see~\eqref{keysquare2}).\\

We now prove the analogues of~\ref{pbtransG2set} and Corollary~\ref{corap} for the topological pullback functor $\fa$.

\begin{lemma}\label{covsofY}
Let $p:E\to Y$ be an object of $\cov\pa{Y}$.
Then, the following is a bijection:
\begin{equation}\label{pi0bij}\begin{split}
\xymatrix@R=0pt{
	\pi_{0}\pa{E}	\ar[r]		&		p^{-1}\pa{y_0}\s G	\\
	C  		\ar@{|-{>}}[r]    &	 F(C)}
\end{split}
\end{equation}
If $C\in\pi_{0}\pa{E}$ and $z\in F(C)$, then $C\cong Y_{\br{ _{z}G}}$ and $F(C)\cong {_{z}}G\b G$.
In particular, $F\pa{Y_{\br{L}}}\cong L\b G$ for each $L<G$.
If $T$ is a transversal for $p^{-1}\pa{y_0}\s G$, then:
\begin{equation}\label{Eisotype}
	E \cong \sum_{z\in T} Y_{\br{_{z}G}}
\end{equation}
The analogous results hold for an object $q:E\to X$ of $\cov\pa{X}$ with $Y$, $G$, and $y_0$ replaced by $X$, $J$, and $x_0$ respectively.
\end{lemma}

\begin{proof}
Let $C\in\pi_{0}\pa{E}$ and let $z\in p^{-1}\pa{y_0}\cap C$.
Then:
\[
	F(C)=p^{-1}\pa{y_0}\cap C = zG
\]
is a transitive right $G$-set, yielding~\eqref{pi0bij}.
As	$p_{\sharp}\pa{\pi_1\pa{C,z}}= {_{z}}G$, we get $C\cong Y_{\br{ _{z}G}}$.
Using the point $z\in zG$, we get the isomorphism (see~\eqref{stabcosettrans}):
\begin{equation}\label{stab_C}\begin{split}
\xymatrix@R=0pt{
	_{z}G\b G	\ar[r]		&		zG	\\
	_{z}G g 		\ar@{|-{>}}[r]    &		z\cdot g}
\end{split}\end{equation}
The next two assertions follow from the first three.
The last assertion holds by the same proof.
\end{proof}

The next lemma is the analogue of~\ref{pbtransG2set} for $\fa$.

\begin{lemma}\label{keylemma}
Let $\br{L}\in SG$ and let $p:Y_{\br{L}}\to Y$ be the corresponding unbased, connected cover.
Let $L\in\br{L}$ be any representative subgroup.
Then, there exists a bijection:
\begin{equation}\label{compspb}
\xymatrix@R=0pt{
		L\b G\s H		\ar[r]^-{\Gamma}	&		\pi_0\pa{\fa\pa{Y_{\br{L}}}}	}
\end{equation}
and:
\begin{equation}\label{isotypecomp}
	\Gamma\pa{LgH}\cong X_{\br{\pa{g^{-1}Lg\, \cap H}'}}
\end{equation}
\end{lemma}

Recall that $K' := \fs^{-1}(K)<J$.

\begin{proof}
There exists $w\in p^{-1}\pa{y_0}$ such that $p_{\sharp}\pa{\pi_{1}\pa{Y_{\br{L}},w}}=L$.
Note that $_{w}G=L$ and $wG=p^{-1}\pa{y_0}$.
By Lemma~\ref{covsofY}, we have the $G$-set isomorphism:
\begin{equation}\label{isoforw}\begin{split}
\xymatrix@R=0pt{
	L\b G	\ar[r]^-{\sigma}			&		wG	\\
	L g 		\ar@{|-{>}}[r]    &		w\cdot g}
\end{split}\end{equation}
By~\eqref{keydiagram}, we have:
\begin{equation}\label{keydiagramYL}\begin{split}
\xymatrix{
    Y_{\br{L}}					\ar@{|-{>}}[r]^-{F}	\ar@{|-{>}}[d]_-{\fa}	&	p^{-1}\pa{y_0}		\ar@{=}[r]		\ar@{|-{>}}[d]_-{\varepsilon}		&
    			wG	\ar@{|-{>}}[d]_-{\fs^{\ast}}	&	\ar[l]_{\sigma}^{\cong}	L\b G	 \ar@{|-{>}}[d]_-{\fs^{\ast}}  \ar@{|-{>}}[dr]^-{\iota^{\ast}}\\
    \fa\pa{Y_{\br{L}}} 	\ar@{|-{>}}[r]^-{F} 											&	\fa(p)^{-1}\pa{x_0} 		& wG		\ar[l]_-{\nu\pa{wG}}^-{\cong}	&		L\b G		\ar[l]_{\fs^{\ast}(\sigma)}^{\cong}	&	\ar@{|-{>}}[l]_-{\lambda^{\ast}}	L\b G}
\end{split}\end{equation}
Consider the bijections (two are $J$-set isomorphisms):
\begin{equation}\label{horbij}\begin{split}
\xymatrix@R=0pt{
	\fa(p)^{-1}\pa{x_0} 		& wG		\ar[l]_-{\nu\pa{wG}}^-{\cong}	&		L\b G		\ar[l]_{\fs^{\ast}(\sigma)}^{\cong}	&	\ar[l]	L\b G\\
	\pa{x_0,w\cdot g}  		&		\ar@{|-{>}}[l] w\cdot g  &	\ar@{|-{>}}[l] Lg		&	\ar@{|-{>}}[l]		Lg}
\end{split}\end{equation}
By Lemma~\ref{covsofY}, \eqref{horbij}, and \ref{surjorbitspacebij}, we have bijections:
\[
\xymatrix@R=0pt{
	\pi_0\pa{\fa\pa{Y_{\br{L}}}}	\ar[r]	&	    \fa(p)^{-1}\pa{x_0}\s J 		&   \ar[l]  wG\s J	&		L\b G\s J		\ar[l]	&	\ar[l]	L\b G\s H\\
								C			\ar@{|-{>}}[r]		&			\pa{x_0,w\cdot g}J  		&		\ar@{|-{>}}[l] \pa{w\cdot g}J  &	\ar@{|-{>}}[l] LgJ		&	\ar@{|-{>}}[l]		LgH}
\]
where $C$ denotes the unique component of $\fa\pa{Y_{\br{L}}}$ containing $\pa{x_0,w\cdot g}J$.
Define $\Gamma\pa{LgH}:=C$.
It remains to prove~\eqref{isotypecomp}.
By Lemma~\ref{covsofY}:
\[
	\Gamma\pa{LgH}\cong X_{\br{_{\pa{x_0,w\cdot g}}J}}
\]
Finally:
\[
	{_{\pa{x_0,w\cdot g}}}J	= {_{w\cdot g}}J = {_{Lg}}J  = \fs^{-1}\pa{{_{Lg}}G} = \fs^{-1}\pa{g^{-1}Lg}  = \pa{g^{-1}Lg\cap H}'
\]
where the first two equalities hold by the isomorphisms in~\eqref{horbij}, the third holds by~\ref{stabilizersgrouppb}, the fourth is clear, and the last holds by definition.
\end{proof}

We remind the reader of the abbreviations~\eqref{abbrev}.

\begin{corollary}[Properties of Topological Pullback]\label{keycor}
The following are consequences of Lemma~\ref{keylemma}:
\begin{enumerate}\setcounter{enumi}{\value{equation}}
\item\label{pbYeqX} $\fa\pa{Y}\cong X$
\item If $L\triangleleft G$, then $\fa\pa{Y_{\br{L}}}\cong c\cdot X_{\br{\pa{L\cap H}'}}$ where $c\geq 1$.
\item\label{pbucY} $\fa\pa{\wt{Y}}\cong \br{G:H}\cdot X_{\br{N}}$ where $N:=\ker{\fs}$.
\item If $\fs$ is injective, then $\fa\pa{\wt{Y}}\cong \br{G:H}\cdot \wt{X}$.
\item If $H\triangleleft G$, then $H\b G\s H=H\b G$ and $\fa\pa{Y_{\br{H}}}\cong \br{G:H}\cdot X$.
\item\label{fssurj} The following are equivalent: (1) $\fs$ is surjective, (2) the pullback of each connected cover is connected, and (3) $\fa\pa{Y_{\br{L}}}\cong X_{\br{(L\cap H)'}}$ for each $L<G$.
\item\label{pbH} $\fa\pa{Y_{\br{H}}}\cong c\cdot X +E$ where $c\geq 1$ and $E$ contains no component isomorphic to $X$ ($E$ may be empty).
\setcounter{equation}{\value{enumi}}
\end{enumerate}
\qed
\end{corollary}

\section{Nullity Zero}\label{s:nullityzero}

Recall that an object of $\cov\pa{Y}$ is \textbf{trivial} provided it is isomorphic to a disjoint union $c\cdot Y$ for some cardinal number $c$.
By~\ref{pbYeqX}, $\fa:\cov\pa{Y}\to\cov\pa{X}$ sends each trivial object to a trivial object, specifically $\fa\pa{c\cdot Y}\cong c\cdot X$.\\

We say that $\fa$ has \textbf{nullity zero} provided only the trivial objects of $\cov\pa{Y}$ pullback to trivial objects of $\cov\pa{X}$.\\

We define nullity zero for algebraic pullback similarly.
Let $h:G_1\to G_2$ be a homomorphism.
A $G_2$-set is \textbf{trivial} provided it is isomorphic to a disjoint union $c\cdot\pa{G_2\b G_2}$ for some cardinal number $c$.
We say $h^{\ast}$ has \textbf{nullity zero} provided only the trivial group sets of $\gset{G_2}$ pullback to trivial group sets of $\gset{G_1}$ (cf.~\ref{hstartG2G2}).\\

Let $L$ be a subgroup of $G$.
The \textbf{normal closure} of $L$ in $G$, denoted $\tn{NC}\pa{G,L}$, is the subgroup of $G$ generated by $g^{-1}Lg$ for all $g\in G$.
Thus, $\tn{NC}\pa{G,L}$ is the smallest normal subgroup of $G$ containing $L$. Following Rose~\cite{rose}, $L$ is \textbf{contranormal} in $G$ provided
$\tn{NC}\pa{G,L}=G$.

\begin{theorem}\label{algnullityzero}
Let $h:G_1\to G_2$ be a homomorphism of groups.
Let $I:=\im{h}$.
Then, $h^{\ast}$ has nullity zero if and only if $I$ is contranormal in $G_2$.
\end{theorem}

\begin{proof}
First, we prove the contrapositive of the forward implication.
Assume that $K:=\tn{NC}\pa{G_2,I}\lneqq G_2$.
As $K\triangleleft G_2$ and $I<K$, $h^{-1}\pa{g^{-1}Kg}=G_1$ for every $g\in G$.
Also, $K\b G_2\s I = K\b G_2$ since:
\[
	KgI=gKI=gK=Kg
\]
So, by~\ref{pbtransG2set} we have $h^{\ast}\pa{K\b G_2} \cong \br{G_2 :K}\cdot \pa{G_1\b G_1}$.
As $K\lneqq G_2$, $K\b G_2$ is not a trivial $G_2$-set.
Hence, $h^{\ast}$ does not have nullity zero.\\

Next, we prove the reverse implication.
As each $G_2$-set is a disjoint union of transitive $G_2$-sets and pullback respects disjoint union,
it suffices to prove that if $L<G_2$ and $h^{\ast}\pa{L\b G_2}$ is trivial, then $L=G_2$.
So, suppose $h^{\ast}\pa{L\b G_2}\cong c\cdot\pa{G_1\b G_1}$ for some cardinal number $c$.
Then~\ref{pbtransG2set} implies that $h^{-1}\pa{g^{-1}Lg} = G_1$ for each $g\in G_2$.
Hence, $I\subset g^{-1}Lg$ for each $g\in G_2$.
In other words:
\[
	g^{-1}Ig \subset L \quad \tn{for each $g\in G_2$}
\]
Therefore, $G_2=\tn{NC}\pa{G_2,I}\subset L$.
Hence, $L=G_2$ as desired.
\end{proof}

\begin{corollary}\label{topnullityzero}
$\fa$ has nullity zero if and only if $H:=\im{\fs}$ is contranormal in $G$.
\end{corollary}

One may prove Corollary~\ref{topnullityzero} in the same way as Theorem~\ref{algnullityzero}, using Lemma~\ref{keylemma} in place of~\ref{pbtransG2set}.
An instructive alternative approach is to deduce Corollary~\ref{topnullityzero} from the statement of Theorem~\ref{algnullityzero} and diagram~\eqref{keysquare2} as follows.\\

Both functors $F$ in diagram~\eqref{keysquare2} are equivalences of categories.
Hence, both are full, faithful, and essentially surjective (see Remark~\ref{remequiv}).
Thus, both are essentially injective.
These observations and the definition of $F$ (see~\eqref{Fequiv}) imply the following:
\begin{equation}\label{Ftrivialobjects}
\begin{alignedat}{1}
	E\cong c\cdot Y 	& \ \  \tn{if and only if} \ \ 	F\pa{E}\cong c\cdot\pa{G\b G} \\
	E\cong c\cdot X 	& \ \  \tn{if and only if} \ \ 	F\pa{E}\cong c\cdot\pa{J\b J}
\end{alignedat}
\end{equation}

The definition of the functor $\varepsilon$ (see~\eqref{keysquare}) implies that:
\begin{equation}\label{vareptriv}
\tn{if $S\cong c\cdot\pa{G\b G}$, then $\varepsilon\pa{S}\cong c\cdot\pa{J\b J}$}
\end{equation}
Therefore, we say $\varepsilon$ has \textbf{nullity zero} provided $\varepsilon$ sends only the trivial $G$-sets to trivial $J$-sets.\\

\begin{lemma}
If $\fa$ has nullity zero and $\fa\pa{E}\cong c\cdot X$, then $E\cong c\cdot Y$.
If $\varepsilon$ has nullity zero and $\varepsilon\pa{S}\cong c\cdot \pa{J\b J}$, then $S\cong c\cdot \pa{G\b G}$.
If $\fs^{\ast}$ has nullity zero and $\fs^{\ast}\pa{S}\cong c\cdot \pa{J\b J}$, then $S\cong c\cdot \pa{G\b G}$.
\end{lemma}

\begin{proof}
Suppose that $\fa$ has nullity zero and $\fa\pa{E}\cong c\cdot X$.
As $\fa$ has nullity zero, $E\cong d\cdot Y$ for some cardinal number $d$.
So:
\[
	c\cdot X \cong \fa\pa{E} \cong \fa\pa{d\cdot Y} \cong d\cdot \fa\pa{Y} \cong d\cdot X
\]
where the last isomorphism used~\ref{pbYeqX}. Thus, $c=d$ as desired.
The other two conclusions are proved similarly, but using~\eqref{vareptriv} and~\ref{hstartG2G2} respectively.
\end{proof}

The proof of the next lemma, left to the reader, is a pleasant exercise using the observations directly above and diagram~\eqref{keysquare2}.

\begin{lemma}\label{nzequivs}
The following are equivalent: (i) $\fa$ has nullity zero, (ii) $\varepsilon$ has nullity zero, and (iii) $\fs^{\ast}$ has nullity zero.
\end{lemma}

Theorem~\ref{algnullityzero} and Lemma~\ref{nzequivs} immediately imply Corollary~\ref{topnullityzero}.
This completes our alternative proof of Corollary~\ref{topnullityzero}.\\

We close this section with several examples where $H$ is a proper, contranormal subgroup of a group $G$.
They show that such subgroups arise naturally both algebraically and topologically.

\begin{example}
The simplest example is $G=\tn{Sym}\pa{n}$, the symmetric group on $n\geq 3$ letters, and $H=\fg{\tau}$, any subgroup of $G$ generated by a transposition $\tau\in G$.
\end{example}

\begin{example}
Let $H\neq\cpa{e}$ be any proper subgroup of a simple group $G$.
For instance, take $G$ to be the alternating group $\tn{Alt}\pa{n}$ on $n$ letters where $n\geq 5$, or 
take $G$ to be Thompson's group $T$, a finitely presented, infinite, simple group~\cite{cannonfloyd}.
See Higman~\cite{higman} for more infinite, simple groups.
\end{example}

\begin{example}
Consider Poincar\'{e}'s integral homology $3$-sphere $\Sigma^{3}$ with fundamental group $G$ isomorphic to the binary icosahedral group $\tn{I}\pa{120}$.
The only nontrivial, proper, normal subgroup of $G$ is its center, which has order $2$.
$G$ also contains cyclic subgroups of orders 3, 4, 5, 6, and 10.
Each of these noncentral cyclic subgroups is thus contranormal in $G$, and arises topologically as $\im\fs$ for a covering map $f:X^3\to \Sigma^3$ by a Lens space $X^3$.
\end{example}

\begin{example}
Let $\bs{Y}$ be the wedge sum of two circles based at the wedge point.
\begin{figure}[h!]
    \centerline{\includegraphics{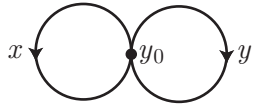}}
    \caption{Wedge sum $\bs{Y}$ of two circles, with fundamental group $G=\p{Y}=\fg{x,y}$ generated by the two embedded loops.}
    \label{wedge}
\end{figure}
Let $G=\p{Y}=\fg{x,y}$ with the canonical generators shown in Figure~\ref{wedge}.
We construct based covering spaces $\bs{X}$ of $\bs{Y}$ using the two building blocks in Figure~\ref{cover_blocks}.
\begin{figure}[h!]
    \centerline{\includegraphics{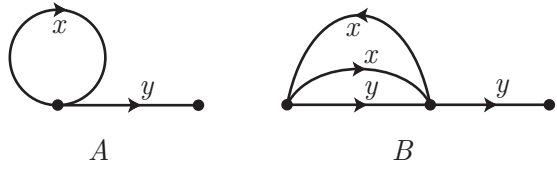}}
    \caption{Building blocks $A$ and $B$ for the construction of covers $\bs{X}$ of $\bs{Y}$. The endpoints of the segments labelled $y$ will comprise the fiber above $y_0$.}
    \label{cover_blocks}
\end{figure}

Let $X_1 X_2 \cdots X_m$ be a finite sequence where $m\geq1$ and each $X_k\in\cpa{A,B}$.
Let $a$ and $b$ denote the number of blocks in the sequence equal to $A$ and $B$ respectively.
The sequence specifies a based covering $f:\bs{X}\to\bs{Y}$ with $a+2b$ sheets as follows; there will be two choices for $x_0$ when $X_1=B$.
The space $X$ is obtained by gluing together copies of $A$ and $B$ (see Figure~\ref{cover_blocks}):
the rightmost endpoint of the horizontal segment in $X_k$ is glued to the leftmost endpoint of the horizontal segment in $X_{k+1}$.
The gluing is cyclic, so $X_m$ is glued to $X_1$ in the same manner.
If $X_1=A$, then we declare the basepoint $x_0$ of $X$ to be the image in $X$ of the leftmost endpoint of the horizontal segment in $X_1$.
If $X_1=B$, then we allow two choices: $x_0$ may be the image in $X$ of either of the two endpoints of either arc labelled $x$ in $X_1$.
Let $H:=\im{\fs}$.
Note that $a+b=m$, $\card{f^{-1}\pa{y_0}}=a+2b$, and $\br{G:H}=a+2b$.\\

For example, the sequence $A$ yields the $1$-sheeted trivial cover of $\bs{Y}$, and
the sequence $BAAAA$, for one of the two choices of $x_0$, yields the $6$-sheeted cover in Figure~\ref{six_sheeted_cover}.
\begin{figure}[h!]
    \centerline{\includegraphics{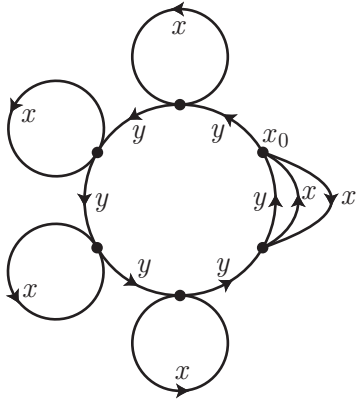}}
    \caption{Six sheeted based covering space $\bs{X}$ of the wedge of two circles $\bs{Y}$.}
    \label{six_sheeted_cover}
\end{figure}

Such a sequence is \textbf{admissible} provided $a\geq1$ and $b\geq1$.
Let  $X_1 X_2 \cdots X_m$ be an admissible sequence.
As $A$ appears in the sequence, $y^{-i}xy^{i}\in H$ for some integer $i$.
As $B$ appears in the sequence, $y^{-j}(xy)y^{j}\in H$ for some integer $j$.
Hence, $x$ and $y$ lie in $\tn{NC}\pa{G,H}$, and $H$ is a contranormal subgroup of $G$ of index $a+2b\geq 3$.
For future reference, we note that if $n:=a+2b$, then:
\[
H\b G = \cpa{H,Hy,Hy^2,\ldots,Hy^{n-1}}
\]
since $H={_{x_0}G}$ for the $G$-monodromy action on the fiber $f^{-1}\pa{y_0}$.\\

Evidently, two based covers arising from this construction are based isomorphic if and only if
their associated sequences are identical and, in case the sequences begin with $B$, the choices of $x_0$ are the same.
An elementary counting argument shows that we have produced:
\[
	 c(n):=\sum_{k=1}^{\left\lfloor{\frac{n-1}{2}}\right\rfloor} \binom{n-k-1}{k} +
	2\sum_{k=1}^{\left\lfloor{\frac{n-3}{2}}\right\rfloor} \binom{n-k-1}{k-1}
\]
pairwise nonisomorphic, based, connected, $n$-sheeted covers of $\bs{Y}$ for each $n\geq 3$.
We double these numbers by simply interchanging the roles of $x$ and $y$ throughout the construction.
Thus, we have produced $2c(n)\geq2$ contranormal subgroups of $G$ of index $n\geq 3$.
By covering space theory, $G$ contains only finitely many subgroups of each finite index $n\in\N$ (see also~\cite{hall}).\\
\end{example}

\begin{remark}
For each fixed $n\geq3$, several of the distinct subgroups of index $n$ in the previous example are $G$-conjugate.
But, at least two $G$-conjugacy classes are represented, and the number of classes represented increases with $n$.
Also, it appears that in $G=\fg{x,y}$, the number of contranormal subgroups of index $n$ greatly exceeds the number of normal subgroups of index $n$, especially as $n$ increases.
\emph{Are contranormal subgroups more common than normal subgroups, at least for most groups (in the sense of Gromov~\cite{gromov}) and finite index $n>2$?}
\end{remark}

\begin{example}
Let $\bs{Y}$ and $G$ be as in the previous example.
Consider bi-infinite sequences $\ldots X_{-1} X_0 X_1 X_2 \ldots $ where each $X_k\in\cpa{A,B}$ and at least one $A$ and one $B$ appear.
Similar to the previous example, each such sequence gives rise to a proper, contranormal subgroup $H$ of $G$, but now of infinite index $\omega$.
This construction yields uncountably many such subgroups, pairwise not $G$-conjugate even.
$G=\fg{x,y}$ also contains uncountably many normal subgroups~\cite[p.~68]{harpe}, hence uncountably many infinite index normal subgroups.
\end{example}

\section{Essential injectivity}\label{s:essinj}

\subsection{Algebraic Equivalent to Essential Injectivity}\label{ss:basic_equivalent}

In this subsection, we prove that $\fs^{\ast}$ and $\fa$ are essentially injective if and only if $\fs$ is surjective.
We begin with a useful observation.

\begin{lemma}\label{congiffinvconj}
Let $h:G_1\to G_2$ be a surjective homomorphism.
Let $A,B<G_2$. Then, $A\equiv_{G_2}B$ if and only if $h^{-1}\pa{A}\equiv_{G_1}h^{-1}\pa{B}$.
\end{lemma}

\begin{proof}
As $h$ is surjective:
\begin{equation}\label{hhinvAeqA}
	h\pa{h^{-1}\pa{A}}=A \quad\tn{and}\quad h\pa{h^{-1}\pa{B}}=B
\end{equation}
The hypothesis for the backward implication is $g^{-1}h^{-1}\pa{A}g=h^{-1}\pa{B}$ for some $g\in G_1$.
Apply $h$ to this hypothesis, and~\eqref{hhinvAeqA} yields $h(g)^{-1}Ah(g)=B$ as desired.\\

For the forward implication, observe that $K:=\ker{h}\triangleleft G_1$ and $K<h^{-1}(A)$ (of course, $K\triangleleft h^{-1}(A)$ and $K\triangleleft h^{-1}(B)$, but these facts are not needed).
By hypothesis, $g^{-1}Ag=B$ for some $g\in G_2$.
As $h$ is surjective, there exists $z\in G_1$ such that $h(z)=g$.
It suffices to prove that $z^{-1}h^{-1}(A)z=h^{-1}(B)$.
The containment ``$\subset$'' is straightforward.
So, let $x\in h^{-1}(B)$.
Then, $h(x)=g^{-1}ag$ for some $a\in A$.
As $h$ is surjective, there exists $y\in G_1$ such that $h(y)=a$.
So:
\[
	h\pa{z^{-1}yz}=g^{-1}ag=h(x)
\]
Hence, $z^{-1}yzk_0=x$ for some $k_0\in K$.
As $K\triangleleft G_1$, $zk_0=kz$ for some $k\in K$.
Thus, $x=z^{-1}(yk)z$ where $yk\in h^{-1}(A)$.
\end{proof}

\begin{lemma}\label{algsurjessinj}
Let $h:G_1\to G_2$ be a homomorphism.
If $h$ is surjective, then $h^{\ast}$ is essentially injective.
\end{lemma}

\begin{proof}
Let $S_1$ and $S_2$ be $G_2$-sets such that $h^{\ast}\pa{S_1}\cong h^{\ast}\pa{S_2}$.
Each of the $G_1$-sets $h^{\ast}\pa{S_1}$ and $h^{\ast}\pa{S_2}$ is isomorphic to a disjoint union of transitive \hbox{$G_1$-sets}.
As $h$ is surjective, the pullback of a transitive $G_2$-set is a transitive \hbox{$G_1$-set} by~\ref{hsurjifftranstotrans}.
Thus, it suffices to prove the case where $S_1$ and $S_2$ are themselves transitive $G_2$-sets.
In this case, $S_1\cong L\b G_2$ and $S_2\cong K\b G_2$ for some subgroups $L$ and $K$ of $G$.
By~\ref{hast}, we have isomorphisms of $G_1$-sets:
\[
	h^{-1}\pa{L}\b G_1 \cong h^{\ast}\pa{L\b G_2} \cong h^{\ast}\pa{S_1} \cong h^{\ast}\pa{S_2} \cong h^{\ast}\pa{K\b G_2} \cong h^{-1}\pa{K}\b G_1
\]
Therefore, $h^{-1}\pa{L}\equiv_{G_1} h^{-1}\pa{K}$ (see above~\eqref{stabcosettrans}).
Hence, $L\equiv_{G_2} K$ by Lemma~\ref{congiffinvconj}, and so $S_1\cong S_2$ (as $G_2$-sets) as desired.
\end{proof}

\begin{lemma}\label{surjessinj}
If $\fs$ is surjective, then $\fa$ is essentially injective.
\end{lemma}

\begin{proof}
We are given that $H=G$.
Suppose $t:\fa\pa{E_1}\to \fa\pa{E_2}$ is an isomorphism where $p_i:E_i\to Y$ is an object of $\cov\pa{Y}$ for $i=1,2$.
In particular, $t$ induces a bijection $\pi_0\pa{\fa\pa{E_{1}}}	\to \pi_0\pa{\fa\pa{E_{2}}}$ and $t$ restricts to an isomorphism $C\to t(C)$ for each component $C$ of $\fa\pa{E_{1}}$.
By~\ref{fssurj}, the pullback of each connected cover of $Y$ is connected.
Hence, pullback induces a bijection $\pi_0\pa{E_i}	\to	\pi_0\pa{\fa\pa{E_{i}}}$ for each $i=1,2$.
Therefore, it suffices to prove the special case where $E_1$ and $E_2$ are connected.
In this case, $E_1\cong Y_{\br{K}}$ and $E_2\cong Y_{\br{L}}$ for some $K,L < G$.
By~\ref{fssurj}:
\[
	X_{\br{K'}}	\cong \fa\pa{Y_{\br{K}}} \cong \fa\pa{E_1} \cong \fa\pa{E_2} \cong \fa\pa{Y_{\br{L}}} \cong X_{\br{L'}}
\]
Thus, $K'\equiv_{J} L'$.
Lemma~\ref{congiffinvconj} implies $K\equiv_{G} L$ and so $E_1\cong E_2$.
\end{proof}

\begin{lemma}\label{notsurjnotessinj}
If $\fs$ is not surjective, then $\fa$ is not essentially injective.
\end{lemma}

\begin{proof}
By~\ref{pbH}:
\[
	\fa\pa{Y_{\br{H}}} \cong c\cdot X + E
\]
where $c\geq 1$ and $E$ (possibly empty) has no component isomorphic to $X$.\\*
Case 1. $c$ is infinite. Then:
\[
	\fa\pa{Y_{\br{H}}+Y} \cong c\cdot X + E + X	\cong c\cdot X + E \cong \fa\pa{Y_{\br{H}}}
\]
since $c+1=c$ (the simplest `infinite swindle'), whereas $Y_{\br{H}}+Y \ncong Y_{\br{H}}$.
The proof of Case 1 is complete.\\*
Case 2. $c$ is finite. Recall that $\omega:=\card{\N}$. Then:
\begin{align*}
	\fa\pa{\omega\cdot Y_{\br{H}}} &\cong	 \omega\cdot c\cdot X + \omega\cdot E\\
	\fa\pa{\omega\cdot Y_{\br{H}}+Y} &\cong \omega\cdot c\cdot X + \omega\cdot E + X\cong \omega\cdot c\cdot X + \omega\cdot E
\end{align*}
since $\omega\cdot c+1=\omega\cdot c$ (another infinite swindle), whereas $\omega\cdot Y_{\br{H}}\ncong \omega\cdot Y_{\br{H}}+Y$ since $H\lneqq G$.
The proof of Case 2 is complete.
\end{proof}

The same argument, but using~\ref{hstarIG2} and~\ref{hstartG2G2}, proves the following.

\begin{lemma}\label{hnotsurjhstarnotessinj}
Let $h:G_1\to G_2$ be a homomorphism.
If $h$ is not surjective, then $h^{\ast}$ is not essentially injective.\qed
\end{lemma}

Lemmas~\ref{algsurjessinj}--\ref{hnotsurjhstarnotessinj} imply the main results of this subsection:

\begin{corollary}
Let $h:G_1\to G_2$ be a homomorphism. Then $h$ is surjective if and only if $h^{\ast}$ is essentially injective.\qed
\end{corollary}

\begin{corollary}
$\fs$ is surjective if and only if $\fa$ is essentially injective.\qed
\end{corollary}

The second case of the proof of Lemma~\ref{notsurjnotessinj} used infinite component covers of $Y$, and both cases used infinite component covers of $X$.
Similar remarks apply to Lemma~\ref{hnotsurjhstarnotessinj} with components replaced by orbits.
The following questions arise.

\begin{question}\label{fcq}
If $\fs$ is not surjective, then do there exist two nonisomorphic, \emph{finite component} covers of $Y$ with isomorphic pullbacks?
Equivalently, do there exist two nonisomorphic $G$-sets with \emph{finite orbit spaces} and isomorphic pullbacks?
\end{question}

\begin{question}\label{fsq}
If $\fs$ is not surjective, then do there exist two nonisomorphic, \emph{finite sheeted} covers of $Y$ with isomorphic pullbacks?
Equivalently, do there exist two nonisomorphic \emph{finite} $G$-sets with isomorphic pullbacks?
\end{question}

\begin{remark}
Recall that topological pullback preserves the number of sheets of a cover (see~\eqref{pbfiber}), but may drastically increase the number of components (see, e.g.,~\ref{pbucY}).
Similarly, algebraic pullback preserves the cardinality of a group set, but may drastically increase the number of orbits (see, e.g.,~\ref{apbeG2}).
\end{remark}

The two questions in Question~\ref{fcq} are equivalent, as are the two questions in Question~\ref{fsq}, by the following lemma.

\begin{lemma}\label{alg_red_lemma}
Consider the four nonhorizontal functors $\fa$, $\varepsilon$, $\fs^{\ast}$, and $\iota^{\ast}$ in diagram~\eqref{keydiagram}.
One of these four functors is essentially injective if and only if all four are essentially injective.
Furthermore, for any fixed (but arbitrary) cardinal numbers $c_1$ and $c_2$, the following are equivalent:
\begin{enumerate}\setcounter{enumi}{\value{equation}}
\item\label{comp_c1c2} There exist covers $E_1$ and $E_2$ of $Y$, with $c_1$ and $c_2$ sheets respectively, such that $E_1\not\cong E_2$ and $\fa\pa{E_1}\cong\fa\pa{E_2}$.
\item There exist $G$-sets $S_1$ and $S_2$, with $c_1$ and $c_2$ elements respectively, such that $S_1\not\cong S_2$ and $\varepsilon^{\ast}\pa{S_1}\cong\varepsilon^{\ast}\pa{S_2}$.
\item There exist $G$-sets $S_1$ and $S_2$, with $c_1$ and $c_2$ elements respectively, such that $S_1\not\cong S_2$ and $\fs^{\ast}\pa{S_1}\cong\fs^{\ast}\pa{S_2}$.
\item\label{alginjequiv} There exist $G$-sets $S_1$ and $S_2$, with $c_1$ and $c_2$ elements respectively, such that $S_1\not\cong S_2$ and $\iota^{\ast}\pa{S_1}\cong\iota^{\ast}\pa{S_2}$.
\setcounter{equation}{\value{enumi}}
\end{enumerate}
Lastly, \ref{comp_c1c2}--\ref{alginjequiv} are equivalent with `sheets' replaced by `components' and `elements' replaced by `orbits'.
\end{lemma}

\begin{proof}
The first conclusion is a pleasant exercise using diagram~\eqref{keydiagram}, Remark~\ref{remequiv}, diagrams~\eqref{factorfs} and~\eqref{keytriangle}, and Lemma~\ref{algsurjessinj}.
The second and third conclusions follow from the proof of the first conclusion, by~\eqref{Fequiv}, and by the bijection~\eqref{pi0bij}.
\end{proof}

Note that~\ref{alginjequiv} is algebraic pullback for inclusion $\iota:H\hookrightarrow G$, which will soon be our focus.\\

We will answer Question~\ref{fcq} in the affirmative (in general).
We will answer Question~\ref{fsq} in the affirmative when $\br{G:H}<\infty$ (in particular, for $\card{G}$ finite).
If $\br{G:H}=\infty$, then Question~\ref{fsq} sometimes has an affirmative answer (consider $X:=\cpa{x_0}\hookrightarrow Y:=S^1$),
and sometimes has a negative answer as the following example shows.

\begin{example}\label{Qexample}
Let $G=\Q$, the additive group of rational numbers (for a presentation of $\Q$, see~\cite[p.~32]{mks}).
Let $H< G$ be any nontrivial, proper subgroup.
Every proper subgroup of $G$ has infinite index~\cite[pp.~61--62]{kurosh}.
In particular, $\br{G:H}=\omega$.
Let $\iota:H\hookrightarrow G$ be inclusion.
By Lemma~\ref{hnotsurjhstarnotessinj},
$\iota^{\ast}:\gset{G}\to\gset{H}$ is not essentially injective.
By~\ref{apbeG2}:
\begin{equation}\label{inf_index_trick}
	\iota^{\ast}\pa{\cpa{e}\b G} \cong \br{G:H}\cdot \pa{\cpa{e}\b H} \cong 2\br{G:H}\cdot \pa{\cpa{e}\b H} \cong \iota^{\ast}\pa{2\cdot\pa{\cpa{e}\b G}}
\end{equation}
whereas $\cpa{e}\b G \not\cong 2\cdot \pa{\cpa{e}\b G}$.
Thus, Question~\ref{fcq} has an affirmative answer for $\iota$.
On the other hand, the only finite, transitive $G$-set up to isomorphism is $G\b G$.
So, the finite $G$-sets up to isomorphism are $n\cdot\pa{G\b G}$ for some $n\in \N_0=\N\cup\cpa{0}$, and
$\iota^{\ast}$ is essentially injective on \emph{finite} $G$-sets.
Thus, Question~\ref{fsq} has a negative answer for $\iota$.
More generally, one may replace $\Q$ with any nontrivial, abelian divisible group,
since any proper subgroup of such a group has infinite index~\cite[p.~59]{harpe}.
Every nontrivial, abelian divisible group is infinitely generated.
\end{example}

\begin{remark}\label{q1infindex}
The argument in~\eqref{inf_index_trick} answers Question~\ref{fcq} in the affirmative when $\br{G:H}=\infty$.
Beyond Example~\ref{Qexample}, we leave Question~\ref{fsq} unexplored when $\br{G:H}=\infty$.
We now focus our attention on answering Question~\ref{fsq} affirmatively
in the case: algebraic pullback for inclusion $\iota:H\hookrightarrow G$ and $\br{G:H}<\infty$.
Evidently, this will answer Question~\ref{fcq} affirmatively when $\br{G:H}<\infty$, and hence in general.
\end{remark}

We close this subsection by reducing to the finite group case.

\begin{lemma}\label{finite_red_lemma}
Suppose $\br{G:H}<\infty$ and $\iota:H\hookrightarrow G$ is not surjective.
Then, there is a commutative diagram of homomorphisms:
\begin{equation}\label{finite_group_diag}\begin{split}
\xymatrix{
    H		\ar@{->>}[d]_-{\pi_0} \ar@{^{(}->}[r]^-{\iota}		&	\ar@{->>}[d]^-{\pi}  G	\\
    H\s K  	\ar@{^{(}->}[r]^-{\iota_0}  	& G\s K }
\end{split}\end{equation}
where $H\s K$ and $G\s K$ are finite groups, and $\iota_0$ is inclusion but is not surjective.
Algebraic pullback yields the commutative diagram of functors:
\begin{equation}\label{finite_red_functors}\begin{split}
\xymatrix{
    \gset{H}	 	&	 \ar[l]_-{\iota^{\ast}}	 \gset{G}	\\
    \gset{\pa{H\s K}} \ar[u]^-{\pi_0^{\ast}} 	  	& \ar[l]_-{\iota_0^{\ast}} \ar[u]_-{\pi^{\ast}} \gset{\pa{G\s K}} }
\end{split}\end{equation}
For any fixed (but arbitrary) cardinal numbers $c_1$ and $c_2$, if:
\begin{enumerate}\setcounter{enumi}{\value{equation}}
\item\label{GKsets} There exist $\pa{G\s K}$-sets $S_1$ and $S_2$, with $c_1$ and $c_2$ elements respectively, such that $S_1\not\cong S_2$ and $\iota_0^{\ast}\pa{S_1}\cong\iota_0^{\ast}\pa{S_2}$,
\setcounter{equation}{\value{enumi}}
\end{enumerate}
then~\ref{alginjequiv} holds.
\end{lemma}

\begin{proof}
$G$ acts on $H\b G$ by right translation.
This action yields the representation:
\[
	\rho:G\to\tn{Sym}\pa{H\b G}
\]
where $\tn{Sym}\pa{H\b G}$ is a finite group (since $\br{G:H}<\infty$).
Evidently:
\[
K:=\ker{\rho}=\bigcap_{g\in G}g^{-1}Hg < H
\]
Hence, $K\triangleleft G$, $K\triangleleft H$,
and $\card{G\s K}<\infty$ (since $G\s K\cong \im{\rho}$).
This readily yields diagram~\eqref{finite_group_diag} satisfying the properties stated there.
Algebraic pullback yields diagram~\eqref{finite_red_functors}.
Lemma~\ref{algsurjessinj} implies that $\pi^{\ast}$ is essentially injective.
So, assuming~\ref{GKsets}, we have $\pi^{\ast}\pa{S_1}\not\cong\pi^{\ast}\pa{S_2}$, but $\iota^{\ast}\pa{\pi^{\ast}\pa{S_1}}\cong \iota^{\ast}\pa{\pi^{\ast}\pa{S_2}}$ by commutativity of~\eqref{finite_red_functors}.
\end{proof}

\subsection{Finite Group Case via Burnside Rings}\label{ss:finitegroups}

Throughout this and the next subsection, $H$ is a subgroup of a finite group $G$.
$\gset{G}$ now denotes the category of \emph{finite} $G$-sets, and similarly for $\gset{H}$.
In this finite setting,
we adhere to convention and write $\res:\gset{G}\to \gset{H} $ (for \emph{restriction}) in place of $\iota^{\ast}$.
Recall that $SG$ denotes the set of $G$-conjugacy classes of subgroups of $G$, and similarly for $SH$.\\

The isomorphism classes of finite $G$-sets form a commutative semi-ring with identity.
Addition is induced by disjoint union.
Multiplication is induced by cartesian product equipped with the diagonal action: $\pa{z_1,z_2}\cdot g:=\pa{z_1\cdot g,z_2\cdot g}$.
The multiplicative identity is $\br{G\b G}$.
The \textbf{Burnside ring} of $G$, denoted $\burn{G}$, is the Grothendieck ring of this semi-ring.
$\burn{G}$ is a commutative ring with identity.
Additively, $\burn{G}$ is a free $\Z$-module with basis $\cpa{\br{L\b G} \mid \br{L}\in SG}$ and rank $n:=\card{SG}$.
Similarly, $\burn{H}$ is a free $\Z$-module with basis $\cpa{\br{K\b H} \mid \br{K}\in SH}$ and rank $m:=\card{SH}$.
Original references for Burnside rings include~\cite{burnside},~\cite{rota}\footnote{``The situation becomes bewildering in problems requiring an enumeration of any of the numerous collections of combinatorial objects which are nowadays coming to the fore.''--Rota.},~\cite{solomon}, and~\cite{gluck}.
Further references include~\cite{bouc},~\cite{yaman}, and~\cite{tomdieck}.\\

Elements of $\burn{G}$ have the form:
\begin{equation}\label{elt_a}
a:=\sum_{\br{L}\in SG} a_{\br{L}} \cdot \br{L\b G}
\end{equation}
where each $a_{\br{L}} \in \Z$.
Let $\burn{G}^{+}\subset\burn{G}$ be the set of isomorphism classes of nonempty, finite $G$-sets.
That is, $\burn{G}^{+}$ contains elements $a\in\burn{G}$ such that
$a_{\br{L}} \geq0$ for all $\br{L}\in SG$ and $a_{\br{L}} >0$ for at least one $\br{L}\in SG$.\\

Recall the $\Z$-module morphisms:
\begin{equation}
\xymatrix{
    \burn{H}	\ar@<-1ex>[r]_-{\ind}	&	\burn{G}	\ar@<-1ex>[l]_-{\res}}
\end{equation}
The \emph{restriction} morphism $\res$ is the natural extension of $\res:\gset{G}\to\gset{H}$ to Burnside rings.
In fact, $\res$ is a unital ring morphism.
The \emph{induction} morphism $\ind$ is defined as follows. Let $S$ be a right $H$-set.
$G$ acts on $S\times G$ on the right by $\pa{z,g}\cdot g':=\pa{z,gg'}$, and
$H$ acts on $S\times G$ on the left by $h\cdot\pa{z,g}:=\pa{x\cdot h^{-1}, hg}$.
By definition, $\ind\pa{S}$ is the quotient $H\b (S\times G)$, often denoted $S\times_{H} G$, equipped with the induced right $G$-action.
In general, $\ind$ is not a ring morphism.\\

Let $\br{L}\in SG$ and $\br{K}\in SH$. Then:
\begin{align}
\label{res_SG}	\res\br{L\b G} &= \sum_{\substack{LgH\in\\ L\b G\s H}} \br{\pa{g^{-1}Lg \, \cap H} \b H} \\
\label{ind_SH}	\ind\br{K\b H} &= \br{K\b G}
\end{align}
For~\eqref{res_SG}, use~\ref{pbtransG2set}. For~\eqref{ind_SH}, consider $\pa{Kh,g}\mapsto Khg$.\\

The \textbf{Burnside algebra} of $G$ is $\Q\burn{G}:=\Q \otimes_{\Z} \burn{G}$.
It is a $\Q$-vector space with basis $\cpa{1\otimes\br{L\b G} \mid \br{L}\in SG}$ and dimension $n$.
We abbreviate $1\otimes\br{L\b G}$ to $\br{L\b G}$.
The $\Z$-module morphisms $\res$ and $\ind$ naturally yield the $\Q$-vector space morphisms:
\begin{equation}
\xymatrix{
    \Q\burn{H}	\ar@<-1ex>[r]_-{\Q\ind}	&	\Q\burn{G}\ar@<-1ex>[l]_-{\Q\res}}
\end{equation}
Note that $\rest{\Q\res}\burn{G}=\res$ and $\rest{\Q\ind}\burn{H}=\ind$.
$\Q\res$ is a unital algebra morphism, while $\Q\ind$ generally is not.

\begin{lemma}\label{res_BGplus}
$\res$ sends $\burn{G}^{+}$ into $\burn{H}^{+}$. In particular, $\burn{G}^{+}$ contains no element in the kernel of $\res$.
\end{lemma}

\begin{proof}
This is immediate by~\eqref{elt_a} and~\eqref{res_SG}.
In particular, if $a_{\br{L}}>0$, then the coefficient of $\br{\pa{L \cap H} \b H}$ in $\res \, a$ is positive.
\end{proof}

The proof of the next lemma is clear.

\begin{lemma}\label{nonneg_decomp}
Let $v\in\burn{G}$. Then $v=a-b$ for unique $a,b\in\burn{G}$ such that: $a_{\br{L}}\geq0$ and $b_{\br{L}}\geq0$ for each $\br{L}\in SG$ and $a$ and $b$ have disjoint support (i.e., $a_{\br{L}}\neq0$ implies $b_{\br{L}}=0$, and $b_{\br{L}}\neq0$ implies $a_{\br{L}}=0$). \qed
\end{lemma}

\begin{lemma}\label{ker_equiv}
There exist $a,b\in\burn{G}^{+}$ such that $a\neq b$ and $\res\,a=\res\,b$ if and only if $\ker{\Q\res}$ is nontrivial.
\end{lemma}

\begin{proof}
For the forward direction, consider $a-b$. For the backward direction, let $u$ be a nontrivial element of $\ker{\Q\res}$.
Then, $cu\in\ker{\res}$ for some $c\in\N$.
Write $cu=a-b$ for unique $a$ and $b$ as in Lemma~\ref{nonneg_decomp}.
It remains to show $a\neq 0$ and $b\neq 0$.
As $u\neq 0$, $a$ and $b$ are not both zero.
If $a=0$, then $0\neq b=-cu\in\burn{G}^{+}\cap\ker{\res}$ contradicting Lemma~\ref{res_BGplus}.
Similarly, $b\neq0$.
\end{proof}

By the previous lemma, we wish to show $\Q\res$ has a nontrivial kernel when $H\lneqq G$.
Of course, this is clear when $\dim{\Q\burn{G}}>\dim{\Q\burn{H}}$.
It is also clear when $H\triangleleft G$ (and $H\neq G$) since then $\res\br{G\b G}=\br{H\b H}$ and $\res\br{H\b G}=\br{G:H}\cdot\br{H\b H}$
by~\ref{hstartG2G2} and~\ref{hstarI}.
However, such arguments do not always apply.

\begin{example}
Let $G=\tn{Alt}(4)$, the alternating group on four letters. Let $H<G$ be the unique subgroup of order 4.
$H$ is isomorphic to $\Z_2 \oplus \Z_2$. Then $\card{SG}=\card{SH}=5$.
By taking $G=\tn{Alt}(4)\oplus \Z_2$ and $H<G$ the unique subgroup of order 8, we get $\card{SG}=12$ and $\card{SH}=16$.
In the previous two examples, $H\triangleleft G$.
The first example with $H$ not normal in $G$ and $\card{SH}\geq\card{SG}$ (in fact, with $\card{SH}>\card{SG}$) is when $G$ has order 96 (group $\fg{96,3}$ in MAGMA notation)
and $H<G$ is a subgroup isomorphic to $\Z_4 \oplus \Z_2 \oplus \Z_2$.
$G$ contains 3 subgroups of order 16, including $H$, and all are $G$-conjugate.
$G$ contains 16 subgroups of order 3 (cf.~\cite{miller}), $\card{SG}=21$, $\card{SH}=27$, and $G$ does not split as a semi-direct product.
More such examples, with $H$ not normal in $G$ and $\card{SH}\geq\card{SG}$, exist with $\card{G}=96$, 128, 144, 160, 168, 192, and so forth.
Such examples with $\card{G}$ odd seem to be less common, the only such with $\card{G}\leq 1000$ having $\card{G}=351$ and $\card{G}=729$.  
\end{example}

Recall that $\Q\burn{G}$ also has a basis of \emph{primitive idempotents} (see~\cite{gluck,bouc}):
\[
	B_G:=\cpa{e^{G}_{L}\in\Q\burn{G}  \mid \br{L}\in SG}
\]
Similarly, $B_H:=\cpa{e^{H}_{K}\in\Q\burn{H} \mid \br{K}\in SH}$ is a basis for $\Q\burn{H}$.
Primitive idempotents are expressed in terms of the basis $\cpa{\br{L\b G} \mid \br{L}\in SG}$ using Gluck's formula~\cite[p.~751]{bouc}.
Namely, if $L<G$ then:
\begin{equation}\label{gluck_formula}
	e^{G}_{L}=\frac{1}{\card{N_G(L)}}\sum_{K<L} \card{K}\mu(K,L)\br{K\b G}
\end{equation}
where $N_G(L)$ is the \textbf{normalizer} of $L$ in $G$ and $\mu$ is the M\"obius function of the poset of \emph{all} subgroups of $G$ (cf.~\cite{pahlings}).
In particular, note that the sum in~\eqref{gluck_formula} is over all subgroups of $L$, not just conjugacy classes of subgroups.

\begin{example}
Let $G=\tn{Sym}(3)$ and $H=\fg{(1,2)}$.
Let $L:=\fg{(1,2,3)}$.
With respect to the indicated ordered bases of $\burn{G}$ and $\burn{H}$ respectively, $\res$ and $\Q\res$ are represented by the matrix:
\[
M= \quad \kbordermatrix{ & \br{\cpa{e}\b G}	&	\br{H\b G} & \br{L\b G} & \br{G\b G}\\
	\br{\cpa{e}\b H}	&	3	&	1	&	1	&	0\\
	\br{H\b H}			  &	0	&	1	&	0	&	1}
\]
and $e^G_G=\br{\frac{1}{2} \ -1 \ -\frac{1}{2} \ \ 1}^{T}$.
\end{example}

\begin{lemma}[Bouc~{\cite[p.~750]{bouc}}]\label{bouc_lemma}
If $H \lneqq G$, then $\Q\res \, e^{G}_{G}=0$.
\end{lemma}

\begin{lemma}\label{eggne0}
$e^G_G\neq 0$.
\end{lemma}

\begin{proof}
By Gluck's formula~\eqref{gluck_formula}, the coefficient of $\br{G\b G}$ in $e^G_G$ equals:
\[
	\frac{1}{\card{G}}\card{G}\mu(G,G) =1
\]
since $\mu(L,L)= 1$ for any subgroup $L<G$, proving the lemma.
In fact, the coefficient of $\br{L\b G}$ in $e^G_G$ must be nonzero for at least one subgroup $H<L\lneqq G$ since $\Q\res \, e^{G}_{G}=0$,
$\Q\res\br{G\b G}=\br{H\b H}$, and by~\eqref{res_SG}.
\end{proof}

\begin{corollary}\label{question_cor}
Question~\ref{fsq} has an affirmative answer when $\br{G:H}<\infty$. Question~\ref{fcq} has an affirmative answer in general.
\end{corollary}

\begin{proof}
For the first conclusion, Lemmas~\ref{alg_red_lemma} and~\ref{finite_red_lemma} reduce the problem to proper inclusion of finite groups.
These cases are handled by Lemmas~\ref{ker_equiv}, \ref{bouc_lemma}, and~\ref{eggne0}.
The second conclusion follows from the first and Remark~\ref{q1infindex}.
\end{proof}

\subsection{Alternative Approach to the Finite Case}\label{ss:finite2}

This subsection identifies a distinguished, $1$-dimensional subspace of $\ker{\Q\res}$ (when $H\lneqq G$), generally very different from $\tn{Span}_{\Q}\cpa{e_G^G}$.
We were led to consider this subspace prior to our awareness of the bases of primitive idempotents and Bouc's result (Lemma~\ref{bouc_lemma}).
We take a moment to motivate this subspace before we prove its existence. By~\eqref{ind_SH} we have:
\begin{align}
\label{im_ind}	\im{\ind} &= \tn{Span}_{\Z}\cpa{\br{K\b G} \mid \br{K}\in SH} < \burn{G} \\
\label{im_Qind}	\im{\Q\ind} &= \tn{Span}_{\Q}\cpa{\br{K\b G} \mid \br{K}\in SH} < \Q\burn{G}
\end{align}

\begin{proposition}\label{inj_SH}
Let $H<G$ where $\card{G}<\infty$.
The restriction of $\res$ to the submodule $\im{\ind}$ is injective.
Equivalently, the restriction of $\Q\res$ to the subspace $\im{\Q\ind}$ is injective.
\end{proposition}

To avoid interruption, and since Proposition~\ref{inj_SH} serves mainly as motivation, we postpone a proof of Proposition~\ref{inj_SH} to later in this subsection.
Recall that $K<G$ is \textbf{$G$-subconjugate} to $L<G$ provided $K$ is $G$-conjugate to a subgroup of $L$.
Proposition~\ref{inj_SH} says that any nontrivial element of $\ker{\Q\res}$ must have a nonzero coefficient
on some $\br{K\b G}$ where $K$ is not $G$-subconjugate to $H$.
If $H\lneqq G$, then sometimes $G$ itself is the only subgroup of $G$ not $G$-subconjugate to $H$.
As $\Q\res\br{G\b G}=\br{H\b H}$, we are led to consider whether $\br{H\b H}=\Q\res(v)$ for some $v\in\im{\Q\ind}$.
If so, then our desired element of $\ker\Q\res$ is $v-\br{G\b G}$ and our distinguished, $1$-dimensional subspace of $\ker{\Q\res}$ is $\tn{Span}_{\Q}\cpa{v-\br{G\b G}}$.
We now prove that indeed this is the case.\\

Let $H<G$.
Recall that $\card{G}$ is finite in this subsection.
Define the set of \textbf{derived subgroups} $DS$ of $H$ in $G$ to be the closure of the initial set $DS=\cpa{H}$ under the operation:
let $K\in DS$ and $g\in G$, replace $DS$ with $DS\cup\cpa{g^{-1}Kg \, \cap H}$.
Clearly, every derived subgroup is a subgroup of $H$.
Let $D:=\cpa{\br{K}_G \mid K\in DS}$ be $G$-conjugacy classes of derived subgroups.
We define:
\begin{align*}
	W &:= \tn{Span}_{\Z}\cpa{\br{K\b G} \mid \br{K}\in D} < \im{\ind} <\burn{G}\\
	\Q W &:= \tn{Span}_{\Q}\cpa{\br{K\b G} \mid \br{K}\in D} < \im{\Q\ind} <\Q\burn{G}
\end{align*}
The relevance of derived subgroups will become clear below in diagram~\eqref{main_burn_ring} and Lemma~\ref{delta_lemma}.
In short, $\br{H\b H}$ will equal $\Q\res(v)$ for a unique $v\in\im{\Q\ind}$, this $v$ will lie in $\Q W$,
and $\Q W < \im{\Q\ind}$ is often a \emph{proper} subspace of $\im{\Q\ind}$ thus narrowing the location of $v$.
For example, if $H\triangleleft G$, then $D=\cpa{\br{H}_G}$ and $v$ will equal $(1/\br{G:H})\cdot \br{H\b G}$.\\ 

Let $L<G$ and $K<H$.
Then (see~\cite{yaman}):
\begin{align}
	\label{residem}	\Q\res\pa{e^{G}_{L}}	&=	\sum_{\substack{\br{J}\in SH\\ J\equiv_{G} L}} e^{H}_{J}\\
	\label{indidem}	\Q\ind\pa{e^{H}_{K}}	&=	\br{\tn{N}_G(K):\tn{N}_H(K)} e^{G}_{K}
\end{align}

Define the standard inner product on $\Q\burn{G}$ for the basis $B_G$ of primitive idempotents by:
\begin{equation}\label{ip_QBG}
	\fg{e^G_K,e^G_L}_G := \begin{cases} 1	&	\tn{if $K\equiv_G L$}\\	0	& \tn{otherwise}	\end{cases}
\end{equation}
Define an inner product on $\Q\burn{H}$ by:
\begin{equation}\label{ip_QBH}
	\fg{e^H_K,e^H_L}_H := \begin{cases} \br{\tn{N}_G(K):\tn{N}_H(K)}	&	\tn{if $K\equiv_H L$}\\	\hspace{3.5em} 0	& \tn{otherwise}	\end{cases}
\end{equation}
If $K\equiv_H L$, then $\card{\tn{N}_H(K)}=\card{\tn{N}_H(L)}$ and $\card{\tn{N}_G(K)}=\card{\tn{N}_G(L)}$.
So, $\fg{-,-}_H$ is indeed symmetric.
It is the standard inner product on $\Q\burn{H}$, for the basis $B_H$, weighted by positive integers.\\

Let $U$ and $V$ be $\Q$-vector spaces equipped with inner products $\fg{-,-}_U$ and $\fg{-,-}_V$ respectively.
Two $\Q$-vector space morphisms:
\[ U \stackrel{S}{\longrightarrow} V \stackrel{T}{\longrightarrow} U \]
are \textbf{adjoint} provided $\fg{Su,v}_V = \fg{u,Tv}_U$ for all $u\in U$ and $v\in V$.

\begin{lemma}\label{lin_alg}
Let $U$ and $V$ be finite dimensional $\Q$-vector spaces. If $S$ and $T$ are adjoint, as in the previous paragraph, then $V=\im{S} \oplus \ker{T}$ and, by symmetry, $U=\im{T} \oplus \ker{S}$.
\end{lemma}

\begin{proof}
If $Su\in\ker{T}$, then:
\[
	0=\langle u,\vec{0} \rangle_U = \fg{u,TSu}_U = \fg{Su,Su}_V
\]
and $Su=\vec{0}$ by definiteness. So, $\im{S} \cap \ker{T} = \{\vec{0}\}$.\\

As $S$ and $T$ are adjoint, $\ker{T}=\pa{\im{S}}^{\bot}$
and $\dim \im{S} =\dim\im{T}$.
Hence:
\begin{align*}
	\dim V	&= \dim \ker{T} + \dim \im{T}\\
					&= \dim \ker{T} + \dim \im{S}
\end{align*}
and the lemma follows.
\end{proof}

\begin{remark}\label{rem_comp}
The previous proof shows that:
\[
	TS\pa{U} = T\pa{V} \quad \tn{and} \quad ST\pa{V}=S\pa{U}
\]
To see this, note that $TS\pa{U}\subset T\pa{V}$ and, as $T$ is injective on $\im{S}$:
\[
	\dim TS\pa{U} =\dim\im{S} =\dim\im{T} =\dim T\pa{V}
\]
\end{remark}

The next lemma says that $\Q\res$ and $\Q\ind$ are adjoint $\Q$-vector space morphisms for the inner products~\eqref{ip_QBG} and~\eqref{ip_QBH}.
The proof, left to the reader, is straightfoward using equations~\eqref{residem}--\eqref{ip_QBH}.

\begin{lemma}\label{adjointmorphisms}
Let $\br{L}\in SG$ and $\br{K}\in SH$. Then:
\[
\fg{\Q\res\pa{e^G_L},e^H_K}_H = \fg{e^G_L,\Q\ind\pa{e^H_K}}_G
\]
\qed
\end{lemma}

\begin{remark}
Lemmas~\ref{lin_alg} and~\ref{adjointmorphisms} immediately prove Proposition~\ref{inj_SH}.
We originally discovered and proved Proposition~\ref{inj_SH} using topological pullback and a direct inductive
argument using components with maximal corresponding subgroups.
We omit the details of this alternative approach and merely mention that it may be of independent interest since, in the finite group case, it may extend to arbitrary covers (over subgroups of $H$) using Zorn's lemma.
It is not clear whether this approach extends to finite sheeted covers (over subgroups of $H$) in the infinite group case.
Inclusion of a point into the circle shows that $\fa$ need not be essentially injective on finite \emph{component} covers (over subgroups of $H$) in the infinite group case.
\end{remark}

Lemmas~\ref{lin_alg} and~\ref{adjointmorphisms} and Remark~\ref{rem_comp} yield the key commutative diagram of $\Q$-vector space morphisms:
\begin{equation}\begin{split}\label{main_burn_ring}
\xymatrix@R=0pt{
    \Q\burn{G}			\ar[r]^-{\Q\res}						&	\Q\burn{H}	\ar[r]^-{\Q\ind}	&	\Q\burn{G}\\
    \rotatebox{90}{$\subset$} & \rotatebox{90}{$\subset$} & \rotatebox{90}{$\subset$}\\
    \im{\Q\ind}	\ar[r]^-{\rest{\Q\res}}_-{\cong}  & \im{\Q\res}	 \ar[r]^-{\rest{\Q\ind}}_-{\cong}	&	\im{\Q\ind} \\
    \rotatebox{90}{$\subset$} & \rotatebox{90}{$\subset$} & \rotatebox{90}{$\subset$}\\
    \Q W	\ar[r]^-{\rest{\Q\res}}_-{\cong}    & \Q\res (\Q W)	 \ar[r]^-{\rest{\Q\ind}}_-{\cong}	&	\Q W  }
\end{split}\end{equation}
Remark~\ref{rem_comp} yields the two middle row isomorphisms.
The lower left morphism is an isomorphism since it is the restriction of the isomorphism directly above it.
For the lower right morphism, call it $\psi$, notice that $\Q\ind \, \Q\res (\Q W) \subset \Q W$ by the definition of derived subgroups and by~\eqref{res_SG} and~\eqref{ind_SH}.
So, $\psi$ maps into $\Q W$ injectively since it is a restriction of the morphism directly above it.
Thus, the bottom row composition $\Q W\to \Q W$ is injective, and hence an isomorphism since $\Q W$ is finite dimensional.
As the lower left morphism is an isomorphism, $\psi$ is an isomorphism as indicated in~\eqref{main_burn_ring}.

\begin{lemma}\label{delta_lemma}
Let $H< G$ where $\card{G}<\infty$ ($H$ need not be a proper subgroup).
Then, there exists a unique $v\in \im{\Q\ind}$ such that $\Q\res(v) =\br{H\b H}$.
Furthermore, $v\in \Q W$.
\end{lemma}

\begin{proof}
Notice that:
\[
\xymatrix@R=0pt{
    \br{G\b G}			\ar@{|-{>}}[r]^-{\Q\res}						&	\br{H\b H}	\ar@{|-{>}}[r]^-{\Q\ind}	&	\br{H\b G}}
\]
where $\br{H\b H}\in \im{\Q\res}$, and $\br{H\b G}\in \Q W$ since $\br{H}\in D$.
The two isomorphisms on the right in diagram~\eqref{main_burn_ring} imply that $\br{H\b H}\in \Q\res (\Q W)$.
The two isomorphisms on the left in diagram~\eqref{main_burn_ring} now yield the desired conclusions.
\end{proof}

\begin{remark}\label{HnormalGDelta}
Of course, $v$ need not lie in $W$ nor in $\im{\ind}$.
For example, if $H$ is a proper, normal subgroup of $G$, then $v=(1/\br{G:H})\cdot \br{H\b G}$ by~\ref{hstarI} and uniqueness of $v$ in Lemma~\ref{delta_lemma}.
\end{remark}

\begin{remark}
Lemma~\ref{delta_lemma} completes our alternative proof of Corollary~\ref{question_cor}.
In particular, Lemma~\ref{delta_lemma} may replace Lemmas~\ref{bouc_lemma} and~\ref{eggne0} in the proof of Corollary~\ref{question_cor} since it provides a nontrivial element $v-\br{G\b G}$ of $\ker\Q\res$ when $H\lneqq G$.
\end{remark}

Lemma~\ref{delta_lemma} motivates the following definition.
Let $H<G$ where $\card{G}<\infty$.
The \textbf{deviation} of $H$ in $G$ is the smallest natural number $\Delta:=\Delta(G,H)$
such that $\Delta\cdot \br{H\b H}$ lies in the image of the composition $\res\, \ind:\burn{H}\to\burn{H}$.
Lemma~\ref{delta_lemma} implies that $\Delta$ exists.
Evidently, $\Delta(G,H)$ is an isomorphism invariant of the pair, and depends only on the $G$-conjugacy class of $H$ in $G$.
The deviation $\Delta(G,H)$ seems to be a compound measure of how ``non-normal'' $H$ is in $G$ and how small (cardinality-wise) $H$ is in $G$.
We find $\Delta(G,H)$ to be a natural and interesting quantity, so we state three conjectures for further study.

\begin{conjectures}
Let $H<G$ where $\card{G}<\infty$.
Let $\Delta=\Delta(G,H)$.
Then:
\begin{enumerate}\setcounter{enumi}{\value{equation}}
\item\label{conj_1} $\Delta=1$ if and only if $H=G$.
\item\label{conj_2} $\br{G:H}$ divides $\Delta$.
\item\label{conj_3} $\Delta$ divides $\card{G}$. 
\setcounter{equation}{\value{enumi}}
\end{enumerate}
\end{conjectures}

Evidently, $H=G$ implies $\Delta=1$, so~\ref{conj_2} implies~\ref{conj_1}.
If $H\triangleleft G$, then $\Delta=\br{G:H}$ by Remark~\ref{HnormalGDelta}.
So, all three conjectures hold when $H\triangleleft G$.
We have verified all three conjectures for thousands of pairs $(G,H)$ using MAGMA.
It should be interesting to find a formula for $\Delta$, to understand relations between $\Delta(G,H)$ and $\Delta(G,H')$ for $H'<H$, and possibly to compare $\Delta$ with $e^G_G$ from the previous subsection.
We suspect that $\Delta$ has intimate relations with certain entries in the table of marks of $G$.\\

We close this section with three examples that display various phenomena.

\begin{example}
If $\Delta=\br{G:H}$, then $H$ need not be normal in $G$.
Consider $G=\tn{Sym}(3)$ and $H=\fg{\tau}$ any subgroup of $G$ generated by a transposition $\tau\in G$.
Then, $\Delta=3$ (use~\ref{pbtransG2set} with $h:H\hookrightarrow G$ inclusion), but $H$ is not normal in $G$.
This example also shows that neither of $\Delta$ and $\card{H}$ need divide the other. 
\end{example}

\begin{example}
Let $G$ be the group $\fg{192,181}$ in MAGMA notation.
$G$ is a nonabelian group of order $192$.
Let $H$ be the $42$nd subgroup of $G$ using MAGMA's intrinsic ordering of {\ttfamily{Subgroups(G)}}$=SG$.
$H$ is a nonabelian group of order $32$.
Let $\iota:H\hookrightarrow G$ be inclusion.
Then:
\begin{itemize}
\item $\iota^{\ast}$ is not essentially injective since $H\lneqq G$.
\item $\card{SG}=46<\card{SH}=47$.
\item $\tn{NC}(G,H)=G$, so $\iota^{\ast}$ has nullity zero.
\item There exist non-$G$-conjugate subgroups $L$ and $K$ of $G$ such that $\iota^{\ast}(L\b G)\cong \iota^{\ast}(K\b G)$.
\end{itemize}
In other words, the last item says that the matrix of $\res$ (with respect to any orderings of $SG$ and $SH$) has two identical columns.
Neither of the subgroups $L$ or $K$ is $G$-subconjugate to $H$ since $\card{H}=32$ and $\card{L}=\card{K}=6$.
In this example, $\Delta(G,H)=12$, so $\br{G:H}$ divides $\Delta$ and $\Delta$ divides $\card{G}$.
\end{example}

\begin{example}\label{zp2example}
Let $p>0$ be prime.
Let $G:=\Z\s p^2\Z$ and let $H:=p \Z \s p^2\Z \triangleleft G$.
For the ordered bases $\pa{\br{\cpa{e}\b G},\br{H\b G},\br{G\b G}}$ and $\pa{\br{\cpa{e}\b H},\br{H\b H}}$ of $\burn{G}$ and $\burn{H}$ respectively,
the matrix of $\res$ is: 
\[
M= \quad \kbordermatrix{ & \br{\cpa{e}\b G}	&	\br{H\b G}	&	\br{G\b G}\\
	\br{\cpa{e}\b H}	&	p	&	0	&	0\\
		\br{H\b H}			&	0	&	p	&	1}
\]
Evidently, $a=\begin{bmatrix} 0	& 1	& 0 \end{bmatrix}^{T}$ and $b= \begin{bmatrix} 0	& 0	& p \end{bmatrix}^{T}$ are the smallest elements in $\burn{G}^{+}$ (for the $L_1$-norm) such that $a\neq b$ and $Ma=Mb$.
These examples, for various primes $p>0$, show that there is no global upper bound on the
number of components or sheets needed to detect failure of essential injectivity.
For other examples, simpler topologically but with infinite fundamental groups, let $p>0$ be prime and consider $f:S^1\to S^1$ given by $z\mapsto z^p$.
\end{example}

\section{Essential surjectivity}\label{s:esssurj}

Groups and group sets are not necessarily finite in this section.
First, observe that if $h:G_1\to G_2$ is a group homomorphism with nontrivial kernel, then $h^{\ast}$ is not essentially surjective since the pullback of no group set is isomorphic to $\cpa{e}\b G_1$ by~\ref{pbtransG2set}.
Thus, we need only consider injective homomorphisms.
Recall diagrams~\eqref{keysquare2}--\eqref{keydiagram}.
Evidently, one of the three functors $\fa$, $\varepsilon$, and $\fa_{\sharp}$ is essentially surjective if and only if all three are essentially surjective.
If $\fs$ is injective, then $\lambda$ is an isomorphism in~\eqref{factorfs}.
So, by~\ref{hisoequiv}, $\lambda^{\ast}$ is an equivalence, and $\iota^{\ast}$ is essentially surjective if and only if $\fs^{\ast}$ is essentially surjective.
Therefore, throughout this section we consider an inclusion homomorphism $\iota:H\hookrightarrow G$.
Call $(G,H)$ an \textbf{essentially surjective pair} provided $\iota^{\ast}$ is essentially surjective.

\begin{proposition}\label{ess_surj_prop}
The pair $(G,H)$ is essentially surjective if and only if for each $K<H$ there exists $L<G$ such that: (i) $G=LH$ and (ii) $L\cap H = K$.
\end{proposition}

\begin{proof}
First, we prove the reverse implication.
Each $H$-set is a disjoint union of transitive $H$-sets, and each transitive $H$-set is isomorphic to $K\b H$ for some $K<H$.
Thus, it suffices to consider $K\b H$ where $K<H$.
By hypothesis, there exists $L<G$ such that $G=LH$ and $L\cap H = K$.
In particular, $L\b G\s H =\cpa{G}$.
By~\ref{pbtransG2set}, $\iota^{\ast}(L\b G) \cong (L\cap H)\b H = K\b H$.\\

Next, let $K<H$. By hypothesis, there exists a $G$-set $S$ such that $\iota^{\ast}(S)\cong K\b H$.
As $K\b H$ is transitive, $S$ is necessarily transitive.
So, $S\cong A \b G$ for some $A<G$, and $\iota^{\ast}(A\b G)\cong K\b H$.
By~\ref{pbtransG2set}, $\card{A\b G \s H}=1$ and $(A\cap H)\b H\cong K\b H$.
The former implies $G=AeH=AH$;
the latter implies $A\cap H \equiv_{H} K$.
Hence, $K=x^{-1}(A\cap H)x=x^{-1}Ax \, \cap H$ for some $x\in H$.
Define $L:=x^{-1}Ax <G$.
So, $L\cap H=K$ and conjugating the identity $G=AH$ by $x$ yields $G=LH$.
\end{proof}

Let $H<G$. A \textbf{complement} of $H$ in $G$ is a subgroup $L<G$ such that $G=LH$ and $L\cap H = \cpa{e}$.
Proposition~\ref{ess_surj_prop} says that a necessary condition for $(G,H)$ to be an essentially surjective pair is that $H$ has a complement in $G$.
Recall a few properties of complements.
Let $L$ be a complement of $H$ in $G$.
Each element of $G$ is uniquely a product $lh$ where $l\in L$ and $h\in H$.
If $G$ is finite, then $\card{G}=\card{L}\cdot\card{H}$ and $G=HL=LH$.
Complements need not be unique nor even $G$-conjugate (consider $G=(\Z\s 2\Z) \oplus (\Z\s 2\Z)$ and $H=(\Z\s 2\Z) \oplus \cpa{e}$).
If $H$ or $L$ is normal in $G$, then $G$ is an \emph{internal semidirect product} of $L$ and $H$.
If $H$ and $L$ are normal in $G$, then $G$ is the \emph{internal direct sum} of $L$ and $H$.

\begin{remark}
If $L$ is a complement of $H$ in $G$, then $G$ is, by definition, an internal \textbf{Zappa-Sz{\'e}p product} of $L$ and $H$ (see~\cite{szep}).
This product generalizes the semidirect product and is variously called the \emph{knit product} or \emph{double crossed product}.
\end{remark}

\begin{example}
$(G,\cpa{e})$ and $(G,G)$ are essentially surjective pairs for any group $G$.
\end{example}

\begin{example}
Let $p>0$ be prime.
Let $G:=\Z\s p^2\Z$ and let $H:=p \Z \s p^2\Z \triangleleft G$.
The pair $(G,H)$ is not essentially surjective since $H$ has no complement in $G$.
Alternatively, the matrix in Example~\ref{zp2example} shows that no $G$-set pulls back to an $H$-set isomorphic to $\cpa{e}\b H$.
Hence, $G$ finite abelian (indeed, finite cyclic) does not imply $(G,H)$ is essentially surjective.
\end{example}

\begin{example}[A Class of Essentially Surjective Pairs]
Let $G$ be an external semidirect product $A\rtimes B$ where $A$ and $B$ are arbitrary groups
such that the (left) action of $B$ on $A$ (denoted $b\cdot a$) preserves each subgroup of $A$ setwise.
Let $H:=A\times\cpa{e} \triangleleft G$.
We show that $(G,H)$ is an essentially surjective pair.
Let $K<H$.
Then, $K=C\times\cpa{e}$ for some $C<A$.
Let $L:=C\times B \subset G = A\rtimes B = A\times B$ (equal as sets).
By assumption, the action of $B$ sends $C$ into itself.
It follows that $L<G$.
To see that $LH=G$, let $(a,b)\in G$.
Then, $(e,b)\in L$, $(b^{-1}\cdot a, e)\in H$, and:
\[
	(a,b)= (e(b\cdot(b^{-1}\cdot a)),be) =  (e,b)(b^{-1}\cdot a, e) \in LH
\] 
as desired.
Finally, $L\cap H=C\times \cpa{e} = K$.
So, Proposition~\ref{ess_surj_prop} implies that $(G,H)$ is an essentially surjective pair.
This class includes all direct products, since these correspond to the case where $B$ acts trivially on $A$.
To see that this class is more general than direct products, let $A$ be cyclic of prime order and let $B$ act on $A$ nontrivially.
For instance, let $A=\Z\s p\Z=\fg{a}$ where $p\geq 3$ is prime, and let $B=\Z\s 2\Z =\fg{b}$ act on $A$ by $b\cdot a^{k}:=a^{-k}$.
If $p=3$, then this particular example is isomorphic to $(\tn{Sym}(3),\fg{(1,2,3)})$.
\end{example}

On the other hand, not every essentially surjective pair arises from a semidirect product splitting,
and not every semidirect product yields an essentially surjective pair, as shown by the next two examples.

\begin{example}
Let $G:=\tn{Sym}(4)$ and $H:=\fg{(1,2,3)}$.
The pair $(G,H)$ is essentially surjective.
$H$ has exactly three complements in $G$, namely the three subgroups of $G$ of order $8$.
These three complements are pairwise $G$-conjugate.
However, neither $H$ nor any of its three complements is normal in $G$.
Hence, $G$ cannot split as an internal semidirect product of $H$ and any subgroup of $G$.
\end{example}

\begin{example}
Let $G$ be the dihedral group of order $8$, namely the subgroup $\fg{(1,2,3,4),(1,3)}$ of $\tn{Sym}(4)$.
Let $H:=\fg{(1,3),(2,4)}$, a Klein $4$-group in $G$.
As $\br{G:H}=2$, $H\triangleleft G$.
$H$ has two complements in $G$, namely $L_1:=\fg{(1,2)(3,4)}$ and $L_2:=\fg{(1,4)(2,3)}$.
In particular, $G$ splits as an internal semidirect product as $H\rtimes L_1$ and as $H\rtimes L_2$.
Nevertheless, $(G,H)$ is not an essentially surjective pair since the subgroups $\fg{(1,3)}$ and $\fg{(2,4)}$ of $H$ have no corresponding subgroup $L<G$ as required by Proposition~\ref{ess_surj_prop}.
\end{example}

One may view Proposition~\ref{ess_surj_prop} as providing an obstruction, for each $K<H$, to the pair $(G,H)$ being essentially surjective.
$K=\cpa{e}$ mandates that $H$ has a complement $L$ in $G$. 
$K=H$ yields no obstruction.

\begin{question}
Which subgroups $K$ of $H$ yield nontrivial obstructions to essential surjectivity?
What obstructions arise from cyclic subgroups $K$ of $H$?
\end{question}

\end{document}